\documentclass{amsart}

\usepackage{amsfonts, amssymb, amsmath, eucal, verbatim, amsthm, amscd, enumerate}
\usepackage{setspace}
\usepackage{fullpage}
\usepackage{
 ulem,
 amsfonts}
 \usepackage{amsmath}
\usepackage{amsmath,esint, amssymb, amsthm,mathbbol,upgreek,exscale}

 \usepackage{mathtools}

\usepackage[colorlinks=true, pdfstartview=FitV, linkcolor=blue, citecolor=blue,
 urlcolor=blue]{hyperref}
\usepackage{color}

\date{}

\newcommand\N{\mathbb{N}} 
\newcommand\Z{\mathbb{Z}} 
\newcommand\R{\mathbb{R}} 
\newcommand\C{\mathbb{C}} 
\newcommand\T{\mathbb{T}} 
\newcommand\D{\mathbb{D}}

\newcommand\De{\mathbb{\Delta}_{\varepsilon}}
\newcommand\EE{\varepsilon}

\theoremstyle{plain}

\numberwithin{equation}{section}
\newtheorem{theorem}{Theorem}[section]
\newtheorem{proposition}[theorem]{Proposition}
\newtheorem{lemma}[theorem]{Lemma}
\newtheorem{remark}[theorem]{Remark}
\newtheorem*{Remarks}{Remarks}

\usepackage[textmath,displaymath,floats,graphics,delayed]{preview}
  \title[]{Relative equilibria with holes for the  surface quasi-geostrophic equations}
 \author[C. Renault]{Coralie Renault}
\address{IRMAR, Universit\'e de Rennes 1\\ Campus de
Beaulieu\\ 35~042 Rennes cedex\\ France}
\email{coralie.renault@univ-rennes1.fr}

\begin{document}

\begin{abstract}
We study the existence of doubly connected rotating patches for the inviscid  surface quasi-geostrophic equation left open in \cite{16}.  By using the approach proposed by \cite{4} we also prove that   close to the annulus the boundaries  are actually analytic curves.   \end{abstract}


\maketitle{}

\section{Introduction}
In this paper we investigate the surface quasi-geostrophic $(\hbox{SQG})$ model  which describes the evolution of the potential temperature $\theta$ according to  the transport equation,
\begin{equation}\label{sqg}
\left\lbrace
\begin{array}{l}
\partial_t \theta + u \cdot \nabla\theta=0, \text{ }(t,x)\in \R_+ \times \R^2, \\
u=-\nabla^\perp(-\Delta)^{-\frac{1}{2}}\theta,\\
\theta_{|t=0}=\theta_0
\end{array}
\right.
\end{equation}
where $u$ refers to the velocity field and $\nabla^\perp=(-\partial_2,\partial_1)$. The  operator   $(-\Delta)^{-\frac{1}{2}}$ is  defined as follows
$$
(-\Delta)^{-\frac{1}{2}}\theta(x)=\frac{1}{2\pi}\int_{\mathbb{R}^2}\frac{ \theta(y)}{|x-y|}dy.
$$

This model is used to study the atmospheric  circulations near the tropopause and the ocean dynamics in the upper layers,    see for instance  \cite{13,21,23}.
This nonlinear  transport equation is more singular than  the vorticity equation for the  2D Euler equations where   the connection between the velocity and the vorticity is given by  the Biot-Savart law
 \[u=-\nabla^\perp(-\Delta)^{-1}\theta.\]
Another  model appearing in the literature which interpolates  between the $(\hbox{SQG})$  and Euler equations  is the $(\hbox{SQG})_\alpha$ model, see \cite{9}, where  the velocity is given by
 \[u=-\nabla^\perp(-\Delta)^{-1+\frac{\alpha}{2}}\theta,\quad \alpha \in (0,2).\]
These equations have been intensively studied during the past few decades and  abundant results have been established   in different topics such as  the well-posedness problem or  the vorticity dynamics. For instance, it is well-known that for Euler equations when  the initial data  $\theta_0$ belongs to $  L^{\infty} \cap L^{1}$  then there is  a unique global weak solution $\theta \in L^{\infty}(\R^+;  L^{\infty} \cap L^{1})$. This theory   fails for  $\alpha> 0$ due to the singularity of the kernel. 
However, the local well-posedness can be elaborated in  the sub-class of the vortex patches as it was shown in \cite{6} and \cite{14}. Recall that  an  initial datum  is a vortex patch when it takes the form $ \chi_D$, which  is the characteristic function of a smooth bounded domain $D$. The solutions keep this structure for a short time, that is, $\theta(t)=\chi_{D_t}$ where $D_t$ is another  domain describing the deformation of the initial one in the complex plane. 
The global existence of these solutions  is an outstanding open problem except for Euler equations in which case Chemin proved in  \cite{7}  the persistence of smooth regularity globally in time.  Note that a significant progress towards settling this problem, for $\alpha$ enough close to zero,  has been done recently in \cite{22bis}. Another direction related to the construction of  periodic global solutions through the bifurcation theory  has been recently investigated.  They correspond to rotating patches also called V-states or relative equilibria. In this setting the  domain of the patch is explicitly given by a pure rotation with uniform angular  velocity, that is,  $D_t=R_{x_0,\Omega t}D$ where $R_{x_0,\Omega t}$ is the planar rotation with the center $x_0$ and the  angle $\Omega t$; the parameter $\Omega$ is the angular velocity.  The  first example of rotating patches goes back for  Euler equation to   Kirchhoff who  discovered that an ellipse of semi-axes $a$ and $b$ rotates uniformly with the angular velocity $\Omega= \frac{ab}{(a^2+b^2)}$; see for instance \cite[p304]{1} and \cite[p 232]{24}. One century later, Deem and Zabusky gave in \cite{11} numerical evidence of the existence of the V-states with $m-$fold symmetry for each integer $m\in \big\{2,3,4,5\big\}$ and afterwards Burbea gave an analytically proof  in \cite{2}. The main idea of the demonstration is to reformulate the V-states equations with the contour dynamics equations, using the conformal parametrization $\Phi$, and to implement some bifurcation arguments. The bifurcation from the ellipses  to countable curves of non symmetric rotating patches  was discussed numerically and analytically in \cite{4,20,21b}. On the other hand we point out that the extension of this study to the $(\hbox{SQG})_\alpha$ was successfully carried out  in \cite{3,15}. Moreover the boundary regularity was achieved in \cite{3,4,20}. \\
The existence of V-states with one hole, also called doubly connected V-states, has been recently explored in \cite{16,17}. To fix the terminology, a patch $\theta_0= \chi_D$ is said to be doubly connected if the domain $D=D_1 \setminus D_2$ with $D_1$ and $D_2$  being two simply connected bounded domains such that the closure $\overline{D_2}$ is strictly embedded in $ D_1$.  The first result on the existence of m-fold symmetric V-states bifurcating from the annulus $\mathbb{A}_b=\lbrace z; b< \vert z \vert <1 \rbrace$ is established in \cite{16}. Roughly speaking, it is shown that for  higher modes $m$ there exist two branches of m-fold symmetric doubly connected  V-states bifurcating from the annulus at explicit eigenvalues $\Omega_m^{\pm}.$ Similar result with more involved computations was obtained for $(\hbox{SQG})_{\alpha}$ model with $ \alpha \in [0,1)$, see \cite{17}. Actually, it is shown that for given $\alpha \in [0,1)$ and $b\in (0,1)$, there exists $N \in \N$ such that for each $m\geq N$ there exists two  curves of $m$-fold doubly connected V-states bifurcating from the annulus $\mathbb{A}_b$ at the angular velocities
\[\Omega_m^{\pm}=\frac{1}{2} \left((1-b^{-\alpha})S_m+(1-b^2){\Lambda}_1(b)\right) \pm \frac{1}{2} \sqrt{{\Delta}_m(\alpha,b)}\]
with
\[{\Delta}_m(\alpha, b)=\left((b^{-\alpha}+1)S_m-(1+b^2){\Lambda}_1(b)\right)^2-4b^2{\Lambda}_m^2(b),\]
\[ {\Lambda}_m(b)\triangleq \frac{1}{b}\int_0^{+ \infty} J_m(bt) J_m(t) \frac{dt}{t^{1-\alpha}}\]
and
\[S_m \triangleq {\Lambda}_{1}(1)-{\Lambda}_m(1). \]
Where $J_m$ refers to the Bessel function of the first kind.\\
The main goal of this paper is to study the same problem for the SQG equation $(\ref{sqg})$ corresponding to $\alpha=1.$ Our aim is twofolds. First we shall establish the existence of doubly connected V-states and second we shall prove that the boundary is analytic. 
The main result of this paper reads as follows.
\begin{theorem}\label{existence}
Let  $b \in (0,1)$, there exists $N \in \N^{\star}\setminus{\lbrace 1\rbrace}$ with the following property: For any integer $m\geq N$ there exist two analytic curves of m-fold doubly connected V-states  for \eqref{sqg} bifurcating from the annulus $\mathbb{A}_b=\lbrace z \in \C, b < \vert z \vert < 1 \rbrace $ at the angular velocities
\begin{equation}\label{Omega}
\Omega_m^{\pm}=\frac{1}{2}\Big[\Big(1-\frac{1}{b}\Big)S_m+(1-b^2)\Lambda_1(b)\Big] \pm \frac{1}{2} \sqrt{\Delta_m(b)}
\end{equation}
where $S_m,\Lambda_m$ and $ \Delta_m$ are defined above by taking $\alpha=1.$
\end{theorem}
\begin{Remarks}
\begin{itemize}
\item For $\alpha=1$, the expression of $S_m$ can be simplified
 and takes the form
 \[S_n = \frac{2}{\pi}\sum_{k=1}^{n-1}\frac{1}{2k+1}\]
\item As we shall see later in the proofs, the number N is defined as the smallest integer such that 
\[ S_N > b \left(\frac{1+b^2}{1+b}\Lambda_1(b)+\frac{2b}{1+b}\Lambda_N(b)\right).\]
\item Our results are in line with results foretold in \cite{16}.
\end{itemize}
\end{Remarks}
Now we shall sketch  the proof of Theorem $\ref{existence}$ which relies on Crandall-Rabinowitz's theorem applied in suitable Banach spaces that capture the analyticity of the boundary. We mention that these spaces were introduced in \cite{4} in order to study  the simply connected V-states.  The first step is to write the boundary equations using the  exterior conformal parametrization of the domains $D_1$ and $D_2$. These conformal mappings  $\Phi_j : \D^c  \rightarrow D_j^c$  have  the following structure
\[ \forall \vert z \vert > 1 \text{, }\Phi_1(z)=z+\sum_{n\in \N} \frac{a_n}{z^{n}} \quad\text{ and  }\quad \Phi_2(z)=bz+\sum_{n\in \N} \frac{c_n}{z^{n}}.\]
 with $\D$ being the unit closed disc.  The Fourier coefficients are supposed to be real meaning  that we look only for the V-states which are symmetric with respect to the real axis. Notice also that when the boundaries  are assumed to be enough smooth then the $\Phi_j$ admit  unique univalent extension up the boundary.
We recall from Section $\ref{Boundary equations}$ that the boundaries of the V-states are subject to the equations: For $j \in \lbrace 1,2 \rbrace$ and  $\omega \in \T$
\begin{align*}
{G}_j(\Omega,\Phi_1,\Phi_2)(\omega)&\triangleq  \textnormal{Im} \left\lbrace  \Big( \Omega\Phi_j(\omega)- S(\Phi_1,\Phi_j)(\omega)+ S(\Phi_2,\Phi_j)(\omega) \Big)\overline{\Phi_j'(\omega)}\overline{\omega} \right\rbrace \\
&=0 
\end{align*}
with
\[ S(\Phi_i,\Phi_j)(\omega)=\fint_{\T} \frac{\tau\Phi_i'(\tau)-\omega \Phi_j'(\omega)}{\vert \Phi_i(\tau)-\Phi_j(\omega)\vert }\frac{d\tau}{\tau}. \]
To apply the bifurcation arguments we make use of the Banach spaces $X^{k+\log}$ and $Y^{k-1}$ that will be fully described in  the subsection \ref{functionspaces}. The main difficulty is to show that the functionals $G_j$ send a  small neighborhood in $X^{k+\log}$ of the trivial solution $(\hbox{Id}, b\hbox{Id})$ to the space $Y^{k-1}.$ This will be done carefully in Section \ref{RegulaT} where additional regularity properties will also be established. The second step is to compute explicitly  the linearized operator of  the vectorial functional $G=(G_1,G_2)$ at the annular solution $(\hbox{Id},b\,\hbox{Id})$. This part is very computational and after using special structures of  the Gauss hypergeometric functions we obtain the following compact expression: Given 
\[h_1(\omega) =\sum_{n=1}^{+ \infty} a_n\overline{\omega}^n \quad \text{ and } \quad h_2(\omega) =\sum_{n=1}^{+ \infty} c_n\overline{\omega}^n \text{, } \omega \in \T \]
we get
 \[D{G}(\Omega,\hbox{Id},b\,\hbox{Id})(h_1,h_2)(\omega)=\frac{i}{2}\sum_{n\geq 1}(n+1)M_{n+1}\left(\begin{array}{c}
a_n \\ 
c_n
\end{array} \right)\left(\omega^{n+1}-\overline{\omega}^{n+1} \right)\]
where the matrix $M_n$ is given for $n\geq 2$ by
\[M_n \triangleq \left(\begin{array}{cc}
\Omega-S_n+b^2\Lambda_1(b) & -b^2\Lambda_n{b} \\ 
b\Lambda_n(b) & b\Omega+S_n-b\Lambda_1(b)
\end{array}\right) .\]
With this explicit formula in hand we find  the values of $\Omega$ leading  to a one dimensional kernel operator. We also check the full conditions required by the Crandall-Rabinowitz's theorem. This discussion will be investigated in detail  in Section \ref{Linea12}.
   
 In what follows, we will need some notations:
 \begin{itemize}
 \item[$\bullet$] The unit disc  and its boundary will be denoted respectively by $\D$ and $\T$.
 \item[$\bullet$] The disc of r radius and centered in $0$ and its boundary will be denoted by $\D_r$ and  $\T_r$.
 \item[$\bullet$] We denote by C any positive constant that may change from line to line.
 \item[$\bullet$] Let $f :\T \rightarrow \C $ be a continuous function. We define its mean value by,
 \[ \fint_{\T} f(\tau ) d\tau \triangleq \frac{1}{2i\pi} \int_{\T} f(\tau) d \tau, \]
 where $d\tau $ stands for the complex integration.
 \item[$\bullet$] Let be $X$ and $Y$ be two normed spaces. We denote by $\mathcal{L}(X,Y)$ the space of all continuous linear maps $T:X \rightarrow Y$ endowed with its usual strong topology.
 \item[$\bullet$] Let $Y$ be a vector space and $R$ be  a subspace, then $Y /R $ denotes the quotient space.
 \end{itemize}
\section{Boundary equations}\label{Boundary equations}
We intend in this section to write down the equations governing the V-states in the doubly connected case. But before doing that we shall recall the  Riemann mapping theorem. To restate this result we need to recall the definition of  \textit{simply connected} domains. Let $\widehat{\C} \triangleq \C \cup \lbrace \infty \rbrace$ denote the Riemann sphere, we say that a domain $U \subset \widehat{\C}$ is \textit{simply connected} if the set $\widehat{\C}\setminus U$ is connected.
\begin{theorem}[Riemann Mapping Theorem]
Let $\D$ denote the unit open ball and $ U \subset \C$ be a simply connected bounded domain. Then there is a unique bi-holomorphic map called also conformal, $\Phi: \C \setminus \overline{\D} \rightarrow \C \setminus \overline{U}$ taking the form 
\[ \Phi(z)=az+\sum_{n \in \N} \frac{a_n}{z^n} \quad \text{ with }\quad a>0.\]
\end{theorem}
 Notice that in this theorem the regularity of the boundary has no effect regarding the existence of the conformal mapping but it contributes in the boundary behavior of the conformal mapping, see for instance \cite{25,28}. \\
Next, we shall move to the equations governing the boundary of the doubly connected V-states.  This can be done in the spirit of the paper \cite{16}. Assume that $\theta_0= \chi_D$ is a rotating patch with an angular velocity $\Omega$ and  such that   $D=D_{1}\setminus D_{2}$ is a doubly connected domain meaning that $D_{1}$ and $D_{2}$ are two simply connected bounded  domains with $D_{2} \subset D_{1}$. Denote by $\Gamma_1$ and $\Gamma_2$ their boundaries, respectively. Then following the same lines of \cite{16} we find that the exterior  conformal mappings $\Phi_1$ and $\Phi_2$ associated to $D_1$ and $D_2$ satisfy the coupled nonlinear equations: For $j\in\{1,2\}, \omega\in\T,$
\begin{eqnarray}\label{tildeG_j}
\nonumber\tilde{G}_j(\Omega,\Phi_1,\Phi_2)(\omega)&\triangleq&  \textnormal{Im} \left\lbrace  \Big( \Omega\Phi_j(\omega)- S(\Phi_1,\Phi_j)(\omega)+ S(\Phi_2,\Phi_j)(\omega) \Big)\overline{\Phi_j'(\omega)}\overline{\omega} \right\rbrace \\
&=&0 
\end{eqnarray}
with
\[ S(\Phi_i,\Phi_j)(\omega)=\fint_{\T} \frac{\tau\Phi_i'(\tau)-\omega \Phi_j'(\omega)}{\vert \Phi_i(\tau)-\Phi_j(\omega)\vert }\frac{d\tau}{\tau}. \]
Notice that we aim at finding   V-states which are small perturbation of the annulus $\mathbb{A}_b$ with $b \in (0,1)$ and therefore the conformal mappings take the form,
\[\forall |z|\geq 1,\quad \Phi_1(z)=z+f_1(z)=z + \sum_{n=1}^{+ \infty} \frac{a_n}{z^n}  \]
and
\[\Phi_2(z)=bz+f_2(z)= b z+ \sum_{n=1}^{+ \infty} \frac{b_n}{z^n}. \]

We shall introduce the functionals
\begin{equation}\label{V-state}
G_j(\Omega,f_1,f_2)\triangleq \tilde{G}_j(\Omega,\Phi_1,\Phi_2) \text{ } j=1,2.
\end{equation}
Then equations of the V-states become,
\[\forall \omega \in \T \text{, } G_j(\Omega,f_1,f_2)(\omega)=0 \text{, } j=1,2.\]
Now we can check that the annulus is a rotating patch for any $\Omega \in \R$. Indeed, 
\[ G_1(\Omega,0,0)(\omega)=\textnormal{Im} \left\lbrace- \overline{\omega} \fint_{\T} \frac{\tau-\omega }{\vert \tau-\omega\vert }\frac{d\tau}{\tau}+ \overline{\omega} \fint_{\T} \frac{b\tau-\omega }{\vert b\tau-\omega\vert }\frac{d\tau}{\tau}\right \rbrace.\]
Using the change of variable $\tau =\omega \xi$ in the last equation we obtain:
\[ {G}_1(\Omega,0,0)(\omega)=\textnormal{Im} \left\lbrace-  \fint_{\T} \frac{\xi-1 }{\vert \xi-1\vert }\frac{d\xi}{\xi}+ \fint_{\T} \frac{b\xi-1 }{\vert b\xi-1\vert }\frac{d\xi}{\xi}\right \rbrace. \]
Now we just observe  that each integral is real. In fact using the parametrization $\xi =e^{i\eta}$ one gets,
\[ \forall a \in (0,1] \text{, } \overline{ \fint_{\T} \frac{a\xi-1 }{\vert a\xi-1\vert }\frac{d\xi}{\xi} }={\frac{1}{2 \pi} \int_0^{2 \pi}\frac{ae^{-i\eta}-1 }{\vert ae^{-i\eta}-1\vert} d\eta}.\]
It suffices now to make again the change of variables  $\eta\mapsto-\eta$.
Hence we find,
\begin{equation*}
\forall \omega \in \T, {G}_1(\Omega,0,0)(\omega)=0.
\end{equation*}
Arguing similarly we also  get 
\begin{equation*}
\forall \omega \in \T, {G}_2(\Omega,0,0)(\omega)=0.
\end{equation*}
\section{Tools}
In this section, we shall gather some useful results that we shall use throughout the paper. First, we  will recall  the Crandall-Rabinowitz's theorem which is the key tool of the proof of our main result. Second, we  shall introduce  different  basic Banach spaces needed in the bifurcation. Last, we  shall collect some important properties on   special functions and  which are  helpful in the subsection $\ref{Linearized operator}$ to get compact formula for the linearized operator.\\
\subsection{Crandall-Rabinowitz's theorem}
We intend now to recall Crandall-Rabinowitz's theorem which is an important  tool in the bifurcation theory and will be used in the proof of Theorem $ \ref{existence} $. Let $ F:\R \times X \rightarrow Y$ be a continuous function with $X$ and $Y$ being two Banach spaces. Assume that $F(\lambda,0)=0$ for any $\lambda\in \R$. Whether or not  close to a trivial solution $(\lambda_0,0)$ one may find a branch of non trivial solutions of the equation $F(\lambda,x)=0$ is the main concern of the bifurcation theory. The following theorem provides sufficient conditions for the bifurcation based on the structure of the linearized  operator  at the point $(\lambda_0,0)$. For more details we refer to  \cite{10, 22}. 
\begin{theorem}\label{CR}
Let $X$,$Y$ be two Banach spaces, V a neighborhood of $0$ in $X$ and let $F:\R \times X \rightarrow Y$ with the following properties:
\begin{itemize}
\item[1] $F(\lambda,0)=0$ for any $\lambda\in \R$.
\item[2] The partial derivatives $F_{\lambda}$, $F_x$ and $F_{\lambda x}$ exist and are continuous.
\item[3] $\textnormal{Ker}(\mathcal{L}_0)$ and $Y/ \textnormal{Im}(\mathcal{L}_0)$ are one-dimensional.
\item[4] Transversality assumption: $ \partial_{\lambda}\partial_x F(0,0)x_0 \notin \textnormal{Im}(\mathcal{L}_0)$, where
\[\textnormal{Ker}(\mathcal{L}_0)=span(x_0) \text{, } \mathcal{L}_0\triangleq \partial_xF(0,0).\]
\end{itemize}
If Z is any complement of $\textnormal{Ker}(\mathcal{L}_0)$ in $X$, then there is a neighborhood $U$ of $(0,0)$ in $\R \times X$, an interval $(-a,a)$, and continuous functions $\phi : (-a,a) \rightarrow \R$, $\psi : (-a,a) \rightarrow Z$ such that $\phi(0)=0$,$\psi(0)=0$ and 
\[F^{-1}(0)\cap U = \Big\lbrace \big(\phi(\xi),\xi x_0+ \xi \psi(\xi) \big); \vert \xi \vert <a \Big\rbrace \cup \Big\lbrace (\lambda,0); (\lambda,0) \in U \Big\rbrace \]
\end{theorem}

\subsection{Function spaces}\label{functionspaces}
 We shall see later the spaces that we shall introduce in this paragraph  will play a central role in the proof of our main theorem. They were first  devised in \cite{3} but with a different representation.
Let $ \varepsilon \in (0,1)$ and introduce the sets
$$ C_{\varepsilon}= \left\lbrace z \in \C \vert \,\,\varepsilon < \vert z \vert <\frac{1}{\varepsilon }\right\rbrace \quad\hbox{and}\quad 
\mathbb{\Delta}_{\varepsilon}=\Big\{ z \in \C \vert\,\, \varepsilon <\vert z \vert \Big\}.
 $$
We denote by $\mathcal{A}_\varepsilon$ the set of holomorphic functions $h$ on $\mathbb{\Delta}_{\varepsilon} $  and such that
$$
\forall z\in \De,\quad  h(z)=\sum_{n\geq1} h_n z^{-n}\quad\hbox{with}\quad  h_n\in \mathbb{R}.
$$
For $m\in \N$ we define $\mathcal{A}_\varepsilon^m$ as the set of functions $h\in \mathcal{A}_\varepsilon$ such that
$$
\forall z\in \De,\quad  h(z)=\sum_{n\geq1} h_n z^{-nm+1}.
$$
Let  $\widehat{\mathcal{A}}_\varepsilon$ be  the set of holomorphic functions $h$ on $ C_\EE$ with the property
\[ \forall z \in C_{\varepsilon} \text{, }h(z) = i\sum_{n=1}^{+ \infty} h_{n}(z^{n}-z^{-n}) \text{, } h_{n}\in \R. \]
For $m\in \N$ we define $\widehat{\mathcal{A}}_\varepsilon^m$ as the set of functions $h\in \widehat{\mathcal{A}}_\varepsilon$ such that,
\[ \forall z \in C_{\varepsilon} \text{, }h(z) = i\sum_{n=1}^{+ \infty} h_{n}(z^{nm}-z^{-nm}) \text{, } h_{n}\in \R. \]
Finally we denote by $\tilde{A}_\varepsilon$ the set of holomorphic functions on $C_{\varepsilon}$ and such that,
\[\forall z \in C_{\varepsilon} \text{, } h(z) = \sum_{n \in \Z} h_{n}z^{n} \text{ with } h_{n}\in \R. \]

 For $k\in\mathbb{N}$  we introduce the spaces,
 \begin{eqnarray*}
X^{k+\log}=  \Bigg\{ h\in \mathcal{A}_\EE, &&\int_{0}^{2\pi} \vert h(\varepsilon e^{i\theta})\vert^2 d\theta<+ \infty,\int_{0}^{2\pi}\vert( \partial_z^kh)(\varepsilon e^{i\theta})\vert^2 d\theta<+ \infty,\\
 &&\Big\| \int_{\T} \frac{(\partial_z^k h)(\varepsilon \tau)-(\partial_z^k h)(\varepsilon \cdot)}{\vert \tau -\cdot\vert} \frac{d\tau}{\tau}\Big\|_{L^2(\T)} < + \infty \Bigg\}
\end{eqnarray*}
and 
\begin{eqnarray*}
X^{k+\log}_m=  X^{k+\log}\cap \mathcal{A}_\varepsilon^m.
\end{eqnarray*}
We also define the spaces,
\[Y^{k-1} =\Bigg\{ h\in\widehat{\mathcal{A}}_\varepsilon \text{, }  \int_{0}^{2\pi} \vert h(\varepsilon e^{i\theta})\vert^2 d\theta<+ \infty \text{, }  \int_{0}^{2\pi}\vert (\partial_z^{k-1}h)( \varepsilon e^{i\theta})\vert^2 d\theta<+ \infty \Bigg\},\]
\[Y^{k-1}_m =Y^{k-1}\cap \widehat{\mathcal{A}}_\varepsilon^m\]
and
\begin{eqnarray*}
\tilde{Y}^{k-1} =\Bigg\{ h\in \tilde{A}_\varepsilon\text{, }&& \int_{0}^{2\pi} \vert h(\varepsilon e^{i\theta})\vert^2 d\theta<+ \infty \text{, }\int_{0}^{2\pi} \Big|h(\varepsilon^{-1}{ e^{i\theta}})\Big|^2 d\theta<+ \infty , \\
&&\int_{0}^{2\pi}\Big| (\partial_z^{k-1}h)({ \varepsilon e^{i\theta}})\Big|^2 d\theta<+ \infty
 \text{, } \int_{0}^{2\pi}\Big|(\partial_z^{k-1}h)( \varepsilon^{-1} e^{i\theta})\Big|^2 d\theta<+ \infty \Bigg\}.
\end{eqnarray*}
Next we shall be concerned with a characterization of the  space $X^{k+\log}$ space in terms of the Fourier coefficients. 
\begin{lemma}\label{caracterisation} Let $k\in \mathbb{N}$ and  $h\in \mathcal{A}_{\varepsilon}$ with $h(z)=\sum_{n\in\mathbb{N^*}}h_n z^{-n}$. Then $h\in X^{k+\log}$ if and only if 
\[\forall \omega\in \mathbb{T},\,h(\omega)= \sum_{n=1}^{+ \infty} h_{n} \overline{\omega}^{n}\quad\hbox{and}\quad   \|h\|_{X^{k+\log}}^2\approx \sum_{n=1}^{+ \infty} \frac{ h_{n}^2 }{\varepsilon^{2(n+k)}}n^{2k} (1+\log(n))^2.\]
\end{lemma}
\begin{proof}
It is easy to see that for $z\in \mathbb{\Delta}_{\varepsilon}$
\[(\partial_z^k h)(z)=\sum_{n=1}^{+ \infty} (-1)^k h_{n} \frac{(n+k-1)!}{(n-1)!}\frac{1}{z^{n+k}}. \]
Hence using the identity $\eqref{c1}$ we get for $\omega\in \T$
\begin{align*}
\fint_{\T} \frac{(\partial_z^k h)(\varepsilon \tau)-(\partial_z^k h)(\varepsilon \omega)}{\vert \tau -\omega \vert} \frac{d\tau}{\tau}
&=\sum_{n=1}^{+ \infty} (-1)^k \frac{h_{n}}{\varepsilon^{n+k}} \frac{(n+k-1)!}{(n-1)!} \fint_{\T} \frac{\overline{\tau}^{n+k}-\overline{\omega}^{n+k}}{\vert \tau -\omega \vert} \frac{d\tau}{\tau}\\
&=\sum_{n=1}^{+ \infty} (-1)^k \frac{h_{n}}{\varepsilon^{n+k}} \frac{(n+k-1)!}{(n-1)!}\overline{\omega}^{n+k} \left[-\frac{2}{\pi}-S_{n+k} \right].
\end{align*}
Therefore we may obtain  the equivalence between the norms since $S_n \sim \log(n)$. 
\end{proof}

\subsection{Hypergeometric functions}
We shall  give basic results on the Gauss hypergeometric functions. The formulae listed below will be crucial in the computations of the linearized operator associated to the V-state equations. Recall that $\forall (a,b,c)\in \R \times \R \times \R \setminus(-\N)$ the hypergeometric function $z \mapsto F(a,b,c;z)$ is defined on the open unit disc $\D$ by the power series
\[ F(a,b,c;z)=\sum_{n=0}^{+ \infty} \frac{(a)_n (b)_n}{(c)_n} \frac{z^n}{n!} \text{, } \forall z \in \D. \]
Here, $(x)_n$ is the Pockhhammer symbol defined by,
\[ (x)_n = \left \lbrace \begin{array}{cc}
1  & n=0 \\ 
x(x+1)\cdots (x+n-1) & n\geq 1.
\end{array} \right.\]
One may easily see that
\[ (x)_n=x(1+x)_{n-1} \text{, } (x)_{n+1}=(x+n)(x)_n. \]
For a future use we recall an integral representation of the hypergeometric function, for instance see \cite{29}. Assume that $c> b >0$, then
\[ F(a,b,c;z)=\frac{\Gamma(c)}{\Gamma(b) \Gamma(c-b)} \int_0^1 x^{b-1} (1-x)^{c-b-1} (1-zx)^{-a}
dx \text{, } \forall z\in  \D.\]
The function $\Gamma: \C \setminus (-\N) \rightarrow \C$ refers to the gamma function which is an analytic continuation to the negative half plane of the usual gamma function defined on the positive half-place $\lbrace \textnormal{Re}(z) >0 \rbrace $ by the integral representation,
\[ \Gamma(z)=\int_0^{+ \infty} t^{z-1} e^{-t} dt.\]
Next, we recall some contiguous functions relations of the hypergeometric series, see \cite{29}.

\begin{equation}\label{eq2}
cF\left(a,b,c;z\right)-cF\left(a+1,b,c;z\right)+bzF\left(a+1,b+1,c+1;z\right)=0
\end{equation}
\begin{equation}\label{eq11}
cF\left(a,b,c;z\right)-cF\left(a,b+1,c;z\right)+azF\left(a+1,b+1,c+1;z\right)=0
\end{equation}
\begin{equation}\label{eq12}
bF\left(a,b+1,c;z\right)-aF\left(a+1,b,c;z\right)+(a-b)\left(a,b,c;z\right)=0
\end{equation}
\begin{equation}\label{eq4}
cF\left(a,b,c;z\right)-(c-b)F\left(a,b,c+1;z\right)-bF\left(a,b+1,c+1;z\right)=0
\end{equation}
%
We end this discussion with recalling Bessel function $J_n$ of the first kind  with  $n \in \N$,
\[ \forall z\in \C,\quad J_n(z)=\sum_{k\geq 0}\frac{(-1)^k}{k!(n+k)!}\left( \frac{z}{2} \right)^{2k+n}.\]
We recall the Sonine-Schafheitlin's formula which hold provided that $0<b<a$ and the integral is convergent, see for example  \cite[p. 401]{29},
\begin{align*}
\int_0^{+ \infty} \frac{J_{\mu}(at)J_{\nu}(bt)}{t^{\lambda}}dt=&\frac{a^{\lambda-\nu-1}b^{\nu}\Gamma(\frac{1}{2}\mu+\frac{1}{2}\nu-\frac{1}{2}\lambda+\frac{1}{2})}{2^{\lambda}\Gamma(\nu+1)\Gamma(\frac{1}{2} \mu+\frac{1}{2}\lambda-\frac{1}{2}\nu+\frac{1}{2})}\\
&\times F\left(\frac{\mu+\nu-\lambda+1}{2},\frac{\nu-\lambda-\mu+1}{2},\nu+1;\frac{b^2}{a^2}\right).
\end{align*}
\section{Regularity of the nonlinear functional}\label{RegulaT}
In this section we are going to check that the functionals $G_j$ seen in \eqref{V-state} are well-defined and satisfy the regularity assumption required by  Crandall-Rabinowitz's theorem. Recall that the exterior domains $\C\backslash D_j$ are parametrized by the conformal mappings $\Phi_j$ whose extension to the boundaries enjoy the following structure,
\[ \forall\, \omega\in\T,\quad \Phi_1(\omega)=\omega + \sum_{n \in \N^{\star}} a_n \overline{\omega}^n =\omega+f_1(\omega) \text{ with } a_n\in \R.\]
\[ \Phi_2(\omega)=b\omega + \sum_{n \in \N^{\star}} c_n \overline{\omega}^n =b\omega+ f_2(\omega) \text{ with  } c_n\in \R.\]
The parameter $b$ belongs to $(0,1)$ which means that we are looking for V-states which are perturbation of the annulus centered at zero and of radius $b$ and $1.$
Recall that the equations of the V-states are given by,
\[ \forall \omega \in \T \text{, } G_j(\Omega,f_1,f_2)(\omega)=0 \text{, } j=1,2,\]
where 
\begin{equation}\label{G_j}
G_j(\Omega,f_1,f_2)(\omega)= \textnormal{Im} \left \lbrace \left(\Omega\Phi_j(\omega)- S(\Phi_1,\Phi_j)(\omega)+ S(\Phi_2,\Phi_j)(\omega)\right)\overline{\Phi_j'(\omega)}\overline{\omega} \right \rbrace
\end{equation}
with
\[ S(\Phi_i,\Phi_j)(\omega)=\fint_{\T} \frac{\tau\Phi_i'(\tau)-\omega \Phi_j'(\omega)}{\vert \Phi_i(\tau)-\Phi_j(\omega)\vert }\frac{d\tau}{\tau}. \]

The study of the regularity of these functionals will be done in several steps. In the first step we shall analyze the existence of the functionals and in the second one establish some strong  regularity.
\subsection{Existence}

 The main result of this section reads as follows.
\begin{proposition}\label{biendef}
For  $j\in \lbrace 1,2 \rbrace$ and for any   $ k \geq 3$, there exists $r \in (0,1)$ such that,

 \[ \begin{array}{lll}
    G_j: &\R\times V_r  \times V_r  &\longrightarrow Y^{k-1}\\
         &(\Omega,f_1,f_2) &\longmapsto G_j(\Omega,f_1,f_2)
    \end{array} \]
    is well-defined. Where $V_r=\big\{f\in X^{k+\log},\,  \text{ }\Vert f \Vert_{X^{k+\log}} \leq r\big\}.$
\end{proposition}
The proof of this result is postponed later and is founded on the following lemma. 
\begin{lemma}\label{anal}
Let  $\varepsilon \in (0,1), j\in \{1 ,2 \}$, $V=b^{j-1}\textnormal{Id}+\tilde{V}$ with $\tilde{V}\in V_r$ and $r$ small enough. Let $h\in V_r,$ then the function
\[ K:\omega\in\T\mapsto  \fint_{\T} \frac{\tau \partial_{\tau}h(\tau)-\omega  \partial_{\tau}h(\omega)}{\vert V(\tau)-V(\omega)\vert}\frac{d \tau}{\tau}\]
can be extended analytically in $C_{\varepsilon}$ to a function $\tilde{K}$ with $\tilde{K}\in \tilde{Y}^{k-1}$. In addition,
\[ \Vert \tilde{K} \Vert _{\tilde{Y}^{k-1}} \leq C \Big(\Vert V\Vert_{H^k(\varepsilon\T)}+\Vert V\Vert_{H^k(\varepsilon^{-1}\T)}\Big) \Vert h \Vert_{X^{k+\log}}.\]
\end{lemma}
Before giving details of the proof we need to  make a comment.
\begin{remark}
Take 
 \[ \begin{array}{lllllll}
    h: &\T &\longrightarrow \C & \text{ and} & \tilde{h}: &C_{\varepsilon} &\longrightarrow \C \\
         &\tau &\longmapsto\sum_{n=1}^{+ \infty}a_n\overline{\tau}^n&&&z&\longmapsto\sum_{n=1}^{+ \infty} \frac{a_n}{z^n}
    \end{array} \]
    then for any $\tau\in \T, \, z\in C_\varepsilon$
    \[  \partial_{\tau}h(\tau)=-\sum_{n=1}^{+ \infty}na_n\overline{\tau}^{n+1}
   \quad \text{ and } \quad \partial_z\tilde{h}(z)=-\sum_{n=1}^{+ \infty}n \frac{a_n}{z^{n+1}}. \]
  Thus,
    \[\partial_{\tau}h= \left. \partial_z\tilde{h}\right|_{\T}.\]
\end{remark}
\begin{proof}
By change of variables, we may write
\[K(\omega)=\omega\fint_{\T} \frac{\tau  \partial_{\tau}h(\tau \omega)-\partial_{\tau}h(\omega)}{\vert V(\tau \omega)-V(\omega)\vert }\frac{d\tau}{\tau}.\]
Our next task is to get a holomorphic  extension of $\omega\mapsto |V(\tau \omega)-V(\omega)|.$ For this aim we write for any $\tau,\omega\in \T,$
\begin{align*}
\vert V(\tau \omega)-V(\omega) \vert^2&=\left (V(\tau \omega)-V(\omega)\right) \left( V({ \overline{\tau}}{\omega}^{-1})-V(\omega^{-1}) \right)\\
&=b^{2(j-1)} \vert \tau -1 \vert^2 \,g(\tau, \omega)\,g(\overline{\tau},\omega^{-1}),
\end{align*}
where $g$ can be extended in a usual way as follows,
\begin{equation}\label{f11}\forall z\in \overline{\mathbb{\Delta}_\varepsilon},\quad g(\tau, z)=1+\frac{\tilde{V}(\tau z)-\tilde{V}(z)}{b^{j-1} z(\tau-1)}. 
\end{equation}
Therefore we get as a by-product,
\begin{equation}\label{maxinf}
\exists C>0, \forall \tau\in \T,\forall z\in \overline{\mathbb{\Delta}_\varepsilon},\quad C^{-1}\leq | g(\tau, z)|\leq C.
\end{equation}
Now we shall use the following estimate,
\begin{align*}\vert\tilde{V}(\tau z)-\tilde{V}(z)\vert&\leq\varepsilon  |\tau-1|{\Vert \partial_z\tilde{V}\Vert_{L^\infty( \varepsilon \T)}}.\end{align*}
This follows from the  mean value theorem combined with the maximum principle for holomorphic functions. Indeed, setting  $\widehat{{V}}(z)=\tilde{V}(\frac{1}{z})$, which is holomorphic in  the disc $D_{\frac{1}{\varepsilon}}=\left\lbrace z \in \C , \vert z \vert < \frac{1}{\varepsilon} \right\rbrace$, we deduce by the mean value theorem that for any  $ z_1,z_2 \in D_{\frac{1}{\varepsilon}}$,
\begin{equation}\label{maxprincipe}
\vert \widehat{{V}}(z_1)-\widehat{{V}}(z_2)\vert \leq\vert z_1 -z_2\vert  \|\partial_{z}\widehat{{V}}\|_{L^\infty(\overline{ D_{\frac{1}{\varepsilon}}})}.
\end{equation}
According to the maximum principle one readily gets
\begin{align*} \|\partial_{z}\widehat{{V}}\|_{L^\infty(\overline{ D_{\frac{1}{\varepsilon}}})}&=
\Vert \partial_z \widehat{{V}} \Vert_{L^{ \infty}( \varepsilon^{-1} \T)}\\
& =\varepsilon^2 \Vert \partial_z \tilde{V} \Vert_{L^{ \infty}(\varepsilon \T)},
\end{align*}
 Applying this inequality with $z_1=\frac{1}{\tau z}$ and $z_2=\frac{1}{z}$ for  $z \in C_{\varepsilon}$ we deduce
\[ \vert \tilde{V}(\tau z)-\tilde{V}(z) \vert \leq \varepsilon \vert {\tau}-1 \vert \Vert \partial_z \tilde{V} \Vert_{L^{ \infty}(\varepsilon \T)} \]
which is the desired inequality. Using  Sobolev embedding $X^{k+\log}\hookrightarrow \hbox{Lip}(\varepsilon\T)$ for $k\geq2$  we find
\begin{equation}\label{estimation}
\vert \tilde{V}(\tau z)-\tilde{V}(z) \vert \leq C\vert {\tau}-1 \vert \Vert \partial_z \tilde{V} \Vert_{X^{k-1+\log}} \leq C\vert {\tau}-1 \vert \Vert  \tilde{V} \Vert_{X^{k+\log}} 
\end{equation}
with $C$ a constant depending on $\varepsilon.$ 

Consequently, one may find small $r$ such that for $\tilde V\in V_r$ the function $z\in C_\varepsilon\mapsto  g(\tau,z) g(\tau, z^{-1})$  is holomorphic and does not cross the negative real axis $\R_{-}$. This allows to define the square root of this latter function,  which remains in turn  holomorphic in the same set $C_\varepsilon.$ Finally, the holomorphic extension of $K$ to  $C_\varepsilon$ could be 
\begin{align*}\tilde{K}(z)&= z\fint_{\T} \frac{\tau  (\partial_zh)(\tau z )- (\partial_zh)(z)}{ b^{j-1} \vert \tau - 1\vert } g(\tau, z)^{-\frac{1}{2}} g(\overline{\tau},\frac{1}{ z})^{-\frac{1}{2}}\frac{d\tau}{\tau}\\
&\triangleq z \fint_{\T} k(z,\tau) d\tau. 
\end{align*}
It remains to check the holomorphic structure of this integral with respect to the complex parameter.  Observe that for fixed $\tau\in \T\backslash\{1\}$ the function  $z\in C_\varepsilon \mapsto k(\tau,z)$ is holomorphic. We also note that the mapping  $ \tau\in \T\backslash\{1\}\mapsto k(\tau,z)$ is bounded uniformly in $z\in C_\varepsilon$. This follows from the estimate
\begin{align*}\vert \partial_z h(\tau z)-\partial_z h(z) \vert & \leq 
\varepsilon \vert \tau-1\vert \Vert \partial_z^2h \Vert_{L^{\infty}(\varepsilon \T )}\\
& \leq C \vert \tau-1\vert \Vert h \Vert_{H^3(\varepsilon \T)}\\
& \leq C \vert \tau-1\vert \Vert h \Vert_{X^{k+\log}}.
\end{align*}
Therefore in view of \eqref{maxinf}, we find a constant $C$ such that for any $(z, \tau)\in \overline{C_{\varepsilon}} \times \T$
\begin{equation}\label{Ineq21} \vert k(z,\tau) \vert \leq C.
\end{equation}
Consequently  $\tilde{K}$ is analytic in the annulus $C_{\varepsilon}$ and therefore it belongs to the class $\tilde A_\varepsilon$.  Hence, it remains to check that $\tilde{K}$ has finite  norm in $\tilde Y^{k-1}$.  
We shall start with the  $L^2$ norm of the inner restriction $\omega\in \T\mapsto \tilde{K}(\varepsilon\omega )$. We observe that 
\[ \tilde{K}(\varepsilon \omega)= \varepsilon \omega\fint_{\T} k(\varepsilon \omega,\tau) d\tau.\]
It is obvious  from \eqref{Ineq21} that
\[ \tilde{K}(\varepsilon \cdot)\in L^{\infty}(\T) \subset L^2(\T)\]
with
\[ \Vert \tilde{K}(\varepsilon \cdot) \Vert_{L^2(\T)} \leq C \Vert V \Vert_{H^2(\varepsilon \T)} \Vert h \Vert_{H^3(\varepsilon \T)}.\]
As to the estimate over the exterior  boundary we proceed in the same way as before and we get
\[ \Vert \tilde{K}(\frac{1}{\varepsilon}\cdot ) \Vert_{L^2(\T)} \leq C \Vert V \Vert_{H^2(\varepsilon \T)} \Vert h \Vert_{H^3(\varepsilon \T)}.\]
Now, we want to control the $L^2$ norm of $\partial_z^{k-1}\tilde{K}(\varepsilon^{\pm} \cdot)$. In what follows, we just give details about $\partial_z^{k-1}\tilde{K}(\varepsilon\cdot)$, we deal with the other term with similar ideas. The computations are very long and we shall focus only on the leading term of $\partial_z^{k-1}\tilde{K}$. From Leibniz formula we may write
\begin{eqnarray*}
\partial_z^{k-1}\tilde{K}(z)&=& z\fint_{\T} \frac{(\partial_z^kh)(\tau z )- (\partial_z^kh)(z)}{ b^{j-1} \vert \tau - 1\vert } g(\tau, z)^{-\frac{1}{2}} g(\overline{\tau},\frac{1}{ z})^{-\frac{1}{2}}\frac{d\tau}{\tau}\\
&+& z\fint_{\T} \frac{(\tau^k-1)  (\partial_z^kh)(\tau z )}{ b^{j-1} \vert \tau - 1\vert } g(\tau, z)^{-\frac{1}{2}} g(\overline{\tau},\frac{1}{ z})^{-\frac{1}{2}}\frac{d\tau}{\tau}\\
&+& z\fint_{\T} \frac{\tau  (\partial_zh)(\tau z )- (\partial_zh)(z)}{ b^{j-1} \vert \tau - 1\vert } \partial_z^{k-1}\Big[g(\tau, z)^{-\frac{1}{2}} g(\overline{\tau},\frac{1}{ z})^{-\frac{1}{2}}\Big]\frac{d\tau}{\tau}+l.o.t.\\
&\triangleq& z K_1(z)+zK_2(z)+zK_3(z)+l.o.t.
\end{eqnarray*}
We shall now check that the terms $K_2$ and $K_3$  can actually be included to  the low order terms. Indeed, for $K_2$ we write according to \eqref{maxinf},
\begin{eqnarray*}
\|K_2(\varepsilon \cdot)\|_{L^\infty(\T)}&\le&C\|\partial_z^{k}h(\varepsilon\cdot)\|_{L^2(\T)}.
\end{eqnarray*}
As to the third term $K_3$ we shall only extract   some significant terms  and the other ones are treated in a similar way. First, it is easy to get

\[ \partial_z\Big[g(\tau, z)^{-\frac{1}{2}}g(\overline{\tau},\frac{1}{ z})^{-\frac{1}{2}}\Big]=-\frac{1}{2}\partial_z \Big(g(\tau, z)g(\overline{\tau},\frac{1}{ z}) \Big)g(\tau, z)^{-\frac{3}{2}}g(\overline{\tau},\frac{1}{ z})^{-\frac{3}{2}}\]
and 
\begin{align*}
\partial_z\Big(g(\tau, z)g(\overline{\tau},\frac{1}{ z})\Big)=&\frac{z(\tau-1)\partial_z\tilde{V}(\tau z)+z \left(\partial_z\tilde{V}(\tau z)-\partial_z\tilde{V}(z) \right)-\left(\tilde{V}(\tau z)-\tilde{V}(z)\right)}{z^2b^{j-1}(\tau-1)}g(\overline{\tau},\frac{1}{ z})\\
&+\frac{z\left(\tilde{V}(\frac{\overline{\tau}}{z})-\tilde{V}(\frac{1}{z}) \right)-(\overline{\tau}-1)\partial_z\tilde{V}(\frac{\overline{\tau}}{z})-\left(\partial_z\tilde{V}(\frac{\overline{\tau}}{z})-\partial_z\tilde{V}(\frac{1}{z})\right)}{b^{j-1}z(\overline{\tau}-1)}g(\tau,z).
\end{align*}
Thus
\begin{eqnarray*}
z\partial_z\Big(g(\tau, z)^{-\frac{1}{2}}g(\overline{\tau},\frac{1}{ z})^{-\frac{1}{2}}\Big)&=&-\frac{1}{2b^{j-1}}\frac{\partial_z\tilde{V}(\tau z)-\partial_z\tilde{V}(z)}{(\tau-1)} g(\tau, z)^{-\frac{3}{2}}g(\overline{\tau},\frac{1}{ z})^{-\frac{1}{2}}\\
&+&\frac{1}{2b^{j-1}}\frac{\partial_z\tilde{V}(\frac{\overline{\tau}}{ z})-\partial_z\tilde{V}(\frac1z)}{(\overline{\tau}-1)} g(\tau, z)^{-\frac{1}{2}}g(\overline{\tau},\frac{1}{ z})^{-\frac{3}{2}}+ l.o.t.
\end{eqnarray*}
Iterating this procedure we find
\begin{eqnarray*}
z\partial_z^{k-1}\Big(g(\tau, z)^{-\frac{1}{2}}g(\overline{\tau},\frac{1}{ z})^{-\frac{1}{2}}\Big)&=&-\frac{1}{2b^{j-1}}\frac{\partial_z^{k-1}\tilde{V}(\tau z)-\partial_z^{k-1}\tilde{V}(z)}{(\tau-1)} g(\tau, z)^{-\frac{3}{2}}g(\overline{\tau},\frac{1}{ z})^{-\frac{1}{2}}\\
&+&\frac{1}{2b^{j-1}}\frac{\partial_z^{k-1}\tilde{V}(\frac{\overline{\tau}}{ z})-\partial_z^{k-1}\tilde{V}(\frac1z)}{(\overline{\tau}-1)} g(\tau, z)^{-\frac{1}{2}}g(\overline{\tau},\frac{1}{ z})^{-\frac{3}{2}}+ l.o.t.
\end{eqnarray*}
It follows that
\begin{align*}
 K_3(\varepsilon \omega)&=-\frac{1}{2b^{j-1}}\fint_{\T} \frac{ \tau \partial_zh(\varepsilon \tau \omega )-\partial_zh(\varepsilon \omega)}{\vert \tau - 1\vert }\frac{\partial_z^{k-1}\tilde{V} (\varepsilon \tau \omega )-\partial_z^{k-1}\tilde{V} (\varepsilon \omega)}{ \tau - 1 } g(\tau, \varepsilon \omega)^{-\frac{3}{2}} g(\overline{\tau},\frac{1}{ \varepsilon \omega})^{-\frac{1}{2}} \frac{d\tau}{\tau}\\
&+\frac{1}{2b^{j-1}}\fint_{\T} \frac{ \tau \partial_zh(\varepsilon \tau \omega )-\partial_zh(\varepsilon \omega)}{\vert \tau - 1\vert }\frac{\partial_z^{k-1}\tilde{V} (\frac{\overline\tau }{ \varepsilon \omega} )-\partial_z^{k-1}\tilde{V} (\frac{1}{\varepsilon \omega})}{\overline\tau - 1 } g(\tau, \varepsilon \omega)^{-\frac{1}{2}} g(\overline{\tau},\frac{1}{ \varepsilon \omega})^{-\frac{3}{2}} \frac{d\tau}{\tau}+l.o.t.
\\\end{align*}
By the definition of H\"older spaces
\[ \vert  \partial_z^{k-1}\tilde{V} (\varepsilon^{\pm 1} \tau \omega )- \partial_z^{k-1}\tilde{V}(\varepsilon^{\pm 1} \omega) \vert \leq \varepsilon^{\pm \frac12}\Vert  \partial_z^{k-1}\tilde{V} \Vert_{C^{\frac{1}{2}}(\varepsilon^{\pm1}\T)} \vert \tau-1\vert^{\frac{1}{2}}.\]
Thanks to
\[ \int_{\T} \frac{1}{ \vert \tau-1\vert^{\frac{1}{2}}} |d\tau| < +\infty \]
combined with \eqref{maxinf} we obtain
\[ \| K_3(\varepsilon \cdot)\|_{L^\infty(\T)} \leq C \Big(\Vert\partial_zh\Vert_{L^{\infty}(\varepsilon\T)}+\Vert\partial_z^2h \Vert_{L^{\infty} (\varepsilon\T)} \Big)\Big(  \Vert\partial_z^{k-1}\tilde{V} \Vert_{C^{\frac{1}{2}}(\varepsilon\T)}+ \Vert\partial_z^{k-1}\tilde{V} \Vert_{C^{\frac{1}{2}}(\varepsilon^{-1}\T)}\Big)+...\]
Hence, using  Sobolev embedding we get 
\[ \Vert  K_3(\varepsilon \cdot)\Vert_{L^{\infty} (\T)} \leq C  \Big(\Vert V \Vert_{H^k(\varepsilon\T)}+\Vert V \Vert_{H^k(\varepsilon^{-1}\T)}\Big) \Vert h \Vert_{H^3(\varepsilon\T)} .\]
Now let us move to the estimate of the term  $K_1$ which is  is the most singular one. For this goal we need  the following lemma.
\begin{lemma}\label{exp}
Let be $\varepsilon \in (0,1)$, $\tilde{V} \in \tilde V_r$ and r be small enough. Define for any   $ \tau \in \T$ and  $ z \in \varepsilon\T \cup \varepsilon^{-1} \T$ 
\[g(\tau, z)=1+\frac{\tilde{V}(\tau z)-\tilde{V}(z)}{b^{j-1}z(\tau-1)} . \]
Then
\[ {g(\tau, z)^{-\frac{1}{2}} }= \left(1+\frac{\partial_{z}\tilde{V}(z)}{b^{j-1}} \right)^{-\frac{1}{2}}+{(\tau-1)} H(\tau,z)\]
where $H(\cdot,\varepsilon^{\pm1} \cdot)\in L^{\infty}(\T \times \T)$ and 
\[ \Vert H(\cdot,\varepsilon^{\pm1} \cdot)\Vert_{ L^{\infty}(\T \times \T)} \leq C \Vert \tilde V\Vert_{H^3(\varepsilon^{\pm1}\T)}. \]
\end{lemma}
\begin{proof}

We shall only  prove the result for $z \in \varepsilon \T$. Similar computations can be done for $z\in \EE^{-1}\T$.  From Taylor expansion at  the second order we find,\[\forall z \in \varepsilon \T, \,  g(\tau, z)=1+\frac{\partial_{z}\tilde{V}(z)}{b^{j-1}}+(\tau-1)H_1(\tau,z),\]
 such that
\[ \vert H_1(\tau,z) \vert \leq C \Vert \partial_z^2 \tilde{V} \Vert_{L^{\infty}( \varepsilon \T)} .\]
Using Sobolev embeddings we get for $k\geq3$ and $(\tau,z)\in \T\times \varepsilon \T$
\[\vert H_1(\tau,z) \vert \leq C \Vert \tilde V \Vert_{X^{k+\log}}.\]
Finally, from standard computations we obtain the identity
\[g(\tau, z)^{-\frac{1}{2}}=\left(1+\frac{\partial_{z}\tilde{V}(z)}{b^{j-1}}\right)^{-\frac{1}{2}}+(\tau-1)H(\tau, z)\]
with
\[ H(\tau, z)=-\frac{H_1(\tau,z)\left(\sqrt{1+\frac{\partial_{z}\tilde{V}(z)}{b^{j-1}}}+\sqrt{1+\frac{\partial_{z}\tilde{V}(z)}{b^{j-1}}+(\tau-1)H_1(\tau,z)}\right)^{-1}}{\sqrt{1+\frac{\partial_{z}\tilde{V}(z)}{b^{j-1}}}\sqrt{1+\frac{\partial_{z}\tilde{V}(z)}{b^{j-1}}+(\tau-1)H_1(\tau,z)} }.\]
One may easily check that 
\[ H(\cdot, \cdot) \in L^{\infty}(\T \times \EE \T)\]
and  the desired  result follows immediately by choosing the radius $r$ small enough.
\end{proof}
Let us now see how to use the preceding lemma for  estimating  $K_1.$
According to this lemma  one may obtain a constant $ C$ depending on $\varepsilon$ and $b$ such that

\begin{align*}
\Vert K_1(\varepsilon \cdot)\Vert_{L^2(\T)}& \leq C \Vert(1+\frac{\partial_{z}\tilde{V}(\varepsilon \cdot)}{b^{j-1}})^{-\frac{1}{2}} (1+\frac{\partial_{z}\tilde{V}(\frac{1}{\varepsilon \cdot})}{b^{j-1}})^{-\frac{1}{2}}\Vert_{L^{\infty}(\T)}\Big\| \fint_{\T} \frac{\partial_z^kh(\varepsilon \tau \cdot )-\partial_z^kh(\varepsilon \cdot)}{\vert \tau-1\vert}\frac{d\tau}{\tau}\Big\|_{L^2(\T)} \\
&+C \Vert \partial_z^kh(\varepsilon \cdot)\Vert_{L^2(\T)} \leq C \Vert h \Vert_{X^{k+\log}}.
\end{align*}
This concludes the proof of  the Lemma \ref{anal}.
\end{proof}
Now, we are in position to give the proof of  the  Proposition $\ref{biendef}$.
\begin{proof}
Note that for any $\omega\in \T$ one has 
\[G_j(\Omega,f_1,f_2)(\omega)=\frac{F_j(\omega)-F_j(\frac{1}{\omega})}{2i}, \]
with
\begin{align*}
F_j(\omega)&=\Omega\Phi_j(\omega)\Phi_j'\left(\frac{1}{\omega}\right)\frac{1}{\omega}-\Phi_j'\left(\frac{1}{\omega}\right)\frac{1}{\omega}\fint_{\T} \frac{\tau \partial_{\tau} \Phi_1(\tau)-\omega \partial_{\tau} \Phi_j(\omega)}{\vert \Phi_1(\tau)-\Phi_j(\omega)\vert }\frac{d\tau}{\tau}\\
&+ \Phi_j' \left(\frac{1}{\omega}\right) \frac{1}{\omega} \fint_{\T} \frac{\tau \partial_{\tau} \Phi_2(\tau)-\omega \partial_\tau \Phi_j(\omega)}{\vert \Phi_2(\tau)-\Phi_j(\omega)\vert }\frac{d\tau}{\tau}.
\end{align*}
We shall prove that $F_j$ belongs to $\tilde{Y}^{k-1}$. The first term of the right-hand side describing the rotation term belongs to that space. The remaining terms are of two kinds: the self-induced terms and the interaction terms. For the first ones we simply  use Lemma $\ref{anal}$ with $h=V=\Phi_j$. As to the interaction terms, the integrand is nowhere singular because the interfaces do not intersect and therefore they are well estimated. We shall briefly give more explanation about this fact.
Take the term
 \label{Ktilde}
\[ \widehat{K}(\omega)\triangleq \fint_{\T} \frac{\tau\Phi_2'(\tau)-\omega \Phi_1'(\omega)}{\vert \Phi_2(\tau)-\Phi_1(\omega)\vert }\frac{d\tau}{\tau}=\omega\fint_{\T} \frac{\tau \Phi_2'(\tau \omega)- \Phi_1'(\omega)}{\vert \Phi_2(\tau \omega)-\Phi_1(\omega)\vert }\frac{d\tau}{\tau}.\]
As before we write, for any $\tau, \omega \in \T$,
\[\vert \Phi_2(\tau \omega)-\Phi_1(\omega)\vert =\vert b\tau-1\vert \Big(\tilde{g}(\tau , \omega) \tilde{g}(\overline{\tau},\omega^{-1})\Big)^{\frac12}\]
where 
\[ \forall z \in \overline{\mathbb{\Delta}_{\varepsilon}} , \, \tilde{g}(\tau, z) = 1+\frac{f_2(\tau z) -f_1(z)}{(b\tau -1)z}.\]
From the maximum principe,
\[ \Big|\frac{f_2(\tau z) -f_1(z)}{(b\tau -1)z} \Big| \leq \varepsilon \frac{ \Vert f_2(\varepsilon \cdot) \Vert_{L^{\infty}(\T) }+\Vert f_1(\varepsilon \cdot) \Vert_{L^{\infty}(\T)}}{1-b}\le \varepsilon \frac{2r}{1-b}\cdot\]
Hence 
\[\widehat{K}(z)=z\fint_{\T} \frac{\tau \Phi_2'(\tau z)-\Phi_1'(z)}{\vert b\tau-1\vert } \tilde{g}(\tau, z)^{-\frac{1}{2}} \tilde{g}(\overline{\tau}, \frac{1}{z})^{-\frac{1}{2}} \frac{d\tau}{\tau}.\]
Note that
\[ 0<1-b\leq \vert b\tau-1 \vert \leq 1+b\]
and 
consequently the integrand is less singular than those of the self-induced terms and thus one can find that 
$\widehat{K}$ is analytic in $C_{\varepsilon}$ and belongs to $ \tilde{Y}^{k-1}.$
At this stage we have shown that $F_j$ belongs to the space $ \tilde{Y}^{k-1} $ and to achieve the proof of the proposition it remains to check that the Fourier coefficients of $G_j(\Omega,f_1,f_2)$ belong to $i\R$. By the assumptions, the Fourier coefficients of $\Phi_j=b^{j-1}Id+f_j$ are real and thus the coefficient of $\overline{\Phi_j'}$ are real too. From the stability of this property under the multiplication and the conjugation we deduce that the Fourier coefficients of $\omega \mapsto \Omega \Phi_j(\omega)\overline{\Phi_j'(\omega)}\overline{\omega}$ are real. To end  the proof we shall check that the Fourier coefficients of  $S(\Phi_i,\Phi_j)$ for $i,j\in \{1 , 2 \}$ are real.
We have
 \[ S(\Phi_i,\Phi_j)(\omega)=\sum_{n \in \Z} a_n \omega^n \text{, } a_n=\fint_{\T} \frac{S(\Phi_i,\Phi_j)(\omega)}{\omega^{n+1}} d \omega= \fint_{\T} \fint_{\T}\frac{\tau \Phi_i'(\tau)-\omega \Phi_j'(\omega)}{\vert \Phi_i(\tau)-\Phi_j(\omega) \vert} \frac{d \tau}{\tau}\frac{d \omega}{\omega^{n+1}}.\]
 The coefficient can also be written in the form
 \[ a_n=\frac{1}{4 \pi^2} \int_0^{2 \pi} \int_0^{2 \pi}\frac{e^{i\theta}\Phi_i'(e^{i\theta})-e^{i\eta}\Phi_j'(e^{i\eta})}{\vert \Phi_i(e^{i\theta})-\Phi_j(e^{i\eta})\vert}e^{-in\eta} d\theta d\eta.\]
By taking the conjugate of $a_n$ and using the properties:
\[ \overline{\Phi_i(e^{i\theta})}=\Phi_i(e^{-i\theta}) \text{, } \overline{\Phi_i'(e^{i\theta})}=\Phi_i'(e^{-i\theta})  \text{ and } \vert z \vert =\vert \overline{z}\vert. \]
One may obtain by change of variable
\begin{align*}
\overline{a_n}=& \frac{1}{4 \pi^2} \int_0^{2 \pi}\int_0^{2 \pi}\frac{e^{-i\theta}\Phi_i'(e^{-i\theta})-e^{-i\eta}\Phi_j'(e^{-i\eta})}{\vert \Phi_i(e^{-i\theta})-\Phi_j(e^{-i\eta})\vert}e^{in\eta} d\theta d\eta\\
=&\frac{1}{4 \pi^2} \int_0^{2 \pi} \int_0^{2 \pi}\frac{e^{i\theta}\Phi_i'(e^{i\theta})-e^{i\eta}\Phi_j'(e^{i\eta})}{\vert \Phi_i(e^{i\theta})-\Phi_j(e^{i\eta})\vert}e^{-in\eta} d\theta d\eta=a_n.
\end{align*}
Consequently the Fourier coefficients of $S(\Phi_i,\Phi_j)$ are real and therefore $G_j(\Omega,\Phi_1,\Phi_2)$ belongs to the space $Y^{k-1}$ and the proof of Proposition \ref{biendef} is now achieved.
\end{proof}

\subsection{Regularity} The goal of this section is to study the  strong regularity of $G_j$ and the main result reads as follows.
\begin{proposition}\label{regularite}
For  $j\in \lbrace 1,2 \rbrace$ and for any $ k \geq 3$, there exists $r \in (0,1)$ such that,
 \[ \begin{array}{lll}
    G_j: &\R\times V_r  \times V_r &\longrightarrow Y^{k-1}\\
         &(\Omega,f_1,f_2) &\longmapsto G_j(\Omega,f_1,f_2)
    \end{array} \]
 is of class $C^1$, where $V_r=\big\{f\in X^{k+\log} |  \text{ }\Vert f \Vert_{X^{k+\log}} \leq r\big\}$.
\end{proposition}
\begin{proof}
To prove that $G_j$ is of class $C^1$  we shall first check the existence of its G\^{a}teaux derivative. Second, we will show that this derivative is strongly continuous, and therefore it will necessary coincide with  the Fr\'echet derivative. This will  answer the $C^1$ regularity.
We split  $G_j$ into two terms, the self-induced term and the interaction term,
\[ G_j(\Omega,f_1,f_2)=O_j(\Omega,f_j)+N_j(f_1,f_2) \text{, }j\in \lbrace 1,2 \rbrace \]
with
\[\forall \omega \in \T, \; O_j(\Omega,f_j)(\omega)\triangleq \textnormal{Im} \left\lbrace \left(\Omega\Phi_j(\omega)+(-1)^j S(\Phi_j,\Phi_j)(\omega)\right) \overline{\omega}\overline{\Phi_j'(\omega)}\right\rbrace\]
and
\[N_j(f_1,f_2)(\omega)\triangleq (-1)^{j-1} \textnormal{Im}\left\lbrace S(\Phi_i,\Phi_j)(\omega)\overline{\omega} \overline{\Phi_j'(\omega)} \right\rbrace \text{, } i \neq j .\]
The G\^{a}teaux derivative of  $G_j$ at $(f_1,f_2)$ in the direction $(h_1,h_2)$ is given by the formula:
\begin{eqnarray}\label{diff}
\nonumber DG_j(\Omega,f_1,f_2)(h_1,h_2)&=&DO_j(\Omega,f_j)h_j+ DN_j(f_1,f_2)(h_1,h_2)\\
\nonumber &\triangleq& \underset{t \rightarrow 0}{ \lim } \frac{1}{t}\left[ O_j(\Omega,f_j+th_j)-O_j(\Omega,f_j)\right]+\underset{t \rightarrow 0}{ \lim } \frac{1}{t}\left[ N_j(f_1+th_1,f_2+th_2)-N_j(f_1,f_2)\right]\\
&=&\frac{d}{dt}|_{ t=0}O_j(\Omega,f_j+th_j)+\left.\frac{d}{dt}\right|_{ t=0}N_j(f_1+th_1,f_2+th_2),
\end{eqnarray}
where the limits are taken in the strong topology of $Y^{k-1}$. Once we have checked the existence of these quantities, it remains  to verify that 
the functions,
\[F_1(t,\omega)\triangleq\frac{1}{t}\Big[ O_j(\Omega,f_j+th_j)(\omega)-O_j(\Omega,f_j)(\omega)\Big]-\left.\frac{d}{dt}\right|_{ t=0} O_j(\Omega,f_j+th_j)(\omega)\]
and
\[F_2(t,\omega)\triangleq\frac{1}{t}\Big[ N_j(\Omega,f_1+th_1,f_2+th_2)(\omega)-N_j(\Omega,f_1,f_2)(\omega)\Big]-\left.\frac{d}{dt}\right|_{ t=0} N_j(\Omega,f_1+th_1,f_2+th_2)(\omega)\]
 can be analytically extended on $C_{\varepsilon}$, and their extension, still denoted by  $F_j$, satisfy
\[ \underset{t\rightarrow 0}{\lim} \Vert F_j(t)\Vert_{Y^{k-1}}=0.\]
The existence of  G\^{a}teaux derivative can be done in a straightforward way and one readily gets
\begin{eqnarray}\label{lj}
\nonumber DO_j(\Omega,f_j)h_j(\omega)&=&\textnormal{Im} \left\lbrace \Omega \left(\Phi_j(\omega)\overline{h_j'(\omega)}+\overline{\Phi_j'(\omega)}h_j(\omega) \right) \overline{\omega}+(-1)^j \overline{h_j'(\omega)}\overline{\omega}\fint_{\T}  \frac{\tau\Phi_j'(\tau)-\omega \Phi_j'(\omega)}{\vert \Phi_j(\omega)-\Phi_j(\tau)\vert }\frac{d\tau}{\tau} \right\rbrace\\
\nonumber&+&(-1)^{j-1}\textnormal{Im} \left\lbrace \overline{\Phi_j'(\omega)}\overline{\omega}\fint_{\T}  \frac{(\tau\Phi_j'(\tau)-\omega \Phi_j'(\omega))\textnormal{Re}\bigg( \Big( \overline{h_j(\tau)-h_j(\omega)}\Big) \Big(\Phi_j(\tau)-\Phi_j(\omega) \Big) \bigg)}{\vert \Phi_j(\tau)-\Phi_j(\omega) \vert^3} \frac{d\tau}{\tau} \right\rbrace \\
&+&(-1)^j\textnormal{Im} \left\lbrace \overline{\Phi_j'(\omega)}\overline{\omega}\fint_{\T}  \frac{\tau h_j'(\tau)-\omega h_j'(\omega)}{\vert \Phi_j(\omega)-\Phi_j(\tau)\vert }\frac{d\tau}{\tau} \right\rbrace
\end{eqnarray}
and
\begin{eqnarray}\label{Nj}
\nonumber DN_j(\Omega,f_1,f_2)(h_1,h_2)(\omega)&=&(-1)^{j-1} \textnormal{Im} \Bigg\{ \overline{h_j'(\omega)}\overline{\omega}\fint_{\T}  \frac{\tau\Phi_i'(\tau)-\omega \Phi_j'(\omega)}{\vert \Phi_i(\omega)-\Phi_j(\tau)\vert }\frac{d\tau}{\tau} \\
\nonumber&+&  \overline{\Phi_j'(\omega)}\overline{\omega} \fint_{\T}  \frac{\tau h_i'(\tau)-\omega h_j'(\omega)}{\vert \Phi_i(\omega)-\Phi_j(\tau)\vert }\frac{d\tau}{\tau} \\
&-&\overline{\Phi_j'(\omega)}\overline{\omega}  \fint_{\T} \frac{[\tau \Phi_i'(\tau)-\omega \Phi_j'(\omega)]\textnormal{Re} \bigg( \Big(\overline{h_i(\tau)-h_j(\omega)} \Big) \Big(\Phi_i(\tau)-\Phi_j(\omega)\Big)\bigg)}{\vert \Phi_i(\tau)-\Phi_j(\omega)\vert^3} \frac{d \tau}{\tau} \Bigg\}.
\end{eqnarray}
First we note that  $F_1(t,\omega)$ can be written in the form 
\[ F_1(t,\omega)=\sum_{l=1}^5 \frac{I_l(t,\omega)-I_l(t,\omega^{-1})}{2i}\]
with
\begin{eqnarray*}
I_1(t,\omega)&=&  \frac{\overline{\Phi_j'(\omega)}\overline{\omega}}{t}\fint_{\T}  [\tau \Phi_j'(\tau)-\omega \Phi_j'(\omega)] \left[ \frac{1}{\Delta^t_{\Phi_j}( h_j)(\tau, \omega)}-\frac{1}{\Delta\Phi_j(\tau, \omega)} \right] \frac{d\tau}{\tau} \\ 
&&+\overline{\Phi_j'(\omega)}\overline{\omega} \fint_{\T}  [\tau \Phi_j'(\tau)-\omega \Phi_j'(\omega)] \left[ \frac{\textnormal{Re}\bigg( \Big(\overline{h_j(\tau)-h_j(\omega)} \Big) \Big(\Phi_j(\tau)-\Phi_j(\omega) \Big) \bigg)}{[\Delta\Phi_j(\tau, \omega)]^3} \right] \frac{d\tau}{\tau}  ,
\end{eqnarray*}  
\begin{eqnarray*}
I_2(t,\omega) &=& \overline{\Phi_j'(\omega)}\overline{\omega}\fint_{\T}  [\tau h_j'(\tau)-\omega h_j'(\omega)] \left[ \frac{1}{\Delta^t_{\Phi_j}( h_j)(\tau, \omega)}-\frac{1}{\Delta\Phi_j(\tau, \omega)} \right] \frac{d\tau}{\tau} , \\
I_3(t,\omega)&=& \overline{h_j'(\omega)}\overline{\omega}\fint_{\T}  [\tau \Phi_j'(\tau)-\omega \Phi_j'(\omega)] \left[ \frac{1}{\Delta^t_{\Phi_j}( h_j)(\tau, \omega)}-\frac{1}{\Delta\Phi_j(\tau, \omega)} \right] \frac{d\tau}{\tau} ,\\
I_4(t,\omega)&=& t\overline{h_j'(\omega)}\overline{\omega}\fint_{\T}  \frac{\tau h_j'(\tau)-\omega h_j'(\omega)}{\Delta^t_{\Phi_j} (h_j)(\tau, \omega)} \frac{d\tau}{\tau} ,\\
I_5(t,\omega)&=& t\Omega \overline{\omega}\overline{h_j'(\omega)}h_j(\omega).
\end{eqnarray*} 
We have use the following notations, 
\[\Delta\Phi_j(\tau, \omega)=\vert \Phi_j(\tau)-\Phi_j(\omega)\vert \]
and
\[\Delta^t_{\Phi_j} (h_j)(\tau,\omega)= \vert \Phi_j(\tau)+th_j(\tau)-\Phi_j(\omega)-th_j(\omega) \vert. \]
First, it is not difficult to check the following limit
\[  \underset{t\rightarrow 0}{\lim} \Vert I_5(t)\Vert_{\tilde{Y}^{k-1}}=0\]
Moreover, if t is small enought, one may use the Lemma \ref{anal} with $h=h_j$ and $V=\Phi_j+th_j$ to etablish
\[  \underset{t\rightarrow 0}{\lim} \Vert I_4(t)\Vert_{\tilde{Y}^{k-1}}=0\]
We have to rewrite the terms $I_1,I_2$ and $I_3$ to compute theirs limits. 
We begin to rewrite one part of the integrand term:
\begin{equation}\label{difffrac}
\frac{1}{\Delta^t_{\Phi_j}( h_j)(\tau,\omega)}-\frac{1}{\Delta\Phi_j(\tau,\omega)}=\frac{\Big(\Delta\Phi_j(\tau,\omega)\Big)^2-\left(\Delta^t_{\Phi_j}( h_j)(\tau,\omega)\right)^2}{\Big(\Delta\Phi_j(\tau,\omega) \Big) \left( \Delta^t_{\Phi_j}( h_j )(\tau,\omega)\right)\left(\Delta^t_{\Phi_j}( h_j)(\tau,\omega)+\Delta\Phi_j(\tau,\omega) \right)}.
\end{equation}
 Then we display the dependency on t in the numerator
\begin{equation}\label{dept}
\Big( \Delta\Phi_j(\tau,\omega)\Big)^2-\left( \Delta^t_{\Phi_j}( h_j)(\tau,\omega)\right)^2=-t\left[2\textnormal{Re}\bigg(\overline{\Big(h_j(\tau)-h_j(\omega)\Big)}\Big(\Phi_j(\tau)-\Phi_j(\omega)\Big)\bigg)\right]-t^2\vert h_j(\tau)-h_j(\omega)\vert^2.
\end{equation}
Moreover, straightforward manipulations lead to the following identity usefull for the term $I_3$
\begin{eqnarray}\label{diffrac2}
&\nonumber \frac{1}{\left(\Delta\Phi_j(\tau, \omega)\right)^3}-\frac{2}{\left(\Delta\Phi_j(\tau,\omega) \right) \left(\Delta^t_{\Phi_j}( h_j)( \tau,\omega)\right) \left(\Delta^t_{\Phi_j} (h_j)(\tau, \omega)+\Delta\Phi_j(\tau,\omega) \right)}=\\
 &\frac{\left[\Delta^t_{\Phi_j} (h_j)(\tau,\omega)\right]^2-\left[\Delta\Phi_j(\tau,\omega) \right]^2}{\left[\Delta\Phi_j(\tau,\omega) \right]^3 \left[\Delta^t_{\Phi_j}( h_j)(\tau,\omega) \right] \left[\Delta^t_{\Phi_j}( h_j)(\tau,\omega)+\Delta\Phi_j(\tau,\omega) \right]}+\frac{\left[\Delta^t_{\Phi_j}( h_j)(\tau,\omega)\right]^2-\left[\Delta\Phi_j(\tau,\omega)\right]^2}{\left(\Delta^t_{\Phi_j}( h_j)(\tau,\omega)\right) \left(\Delta\Phi_j(\tau,\omega)\right)^2 \left(\Delta^t_{\Phi_j}( h_j) (\tau,\omega)+ \Delta\Phi_j(\tau,\omega)\right)^2}. 
\end{eqnarray}
Thanks to $(\ref{difffrac})$,$(\ref{dept})$,we rewrite the terms $I_2$ and $I_3$ .
\begin{eqnarray*}
I_2(t,\omega) &=&-t^2 \overline{\Phi_j'(\omega)}\overline{\omega}\fint_{\T}   \frac{\left[\tau h_j'(\tau)-\omega h_j'(\omega) \right] \vert h_j(\tau)-h_j(\omega)\vert^2}{\Big(\Delta\Phi_j(\tau,\omega) \Big) \left( \Delta^t_{\Phi_j}( h_j )(\tau,\omega)\right)\left(\Delta^t_{\Phi_j}( h_j)(\tau,\omega)+\Delta\Phi_j(\tau,\omega) \right)} \frac{d\tau}{\tau} , \\
&&-t\overline{\Phi_j'(\omega)}\overline{\omega}\fint_{\T}  \frac{ \left[\tau h_j'(\tau)-\omega h_j'(\omega) \right]\left[2\textnormal{Re}\bigg(\overline{\Big(h_j(\tau)-h_j(\omega)\Big)}\Big(\Phi_j(\tau)-\Phi_j(\omega)\Big)\bigg)\right]}{\Big(\Delta\Phi_j(\tau,\omega) \Big) \left( \Delta^t_{\Phi_j}( h_j )(\tau,\omega)\right)\left(\Delta^t_{\Phi_j}( h_j)(\tau,\omega)+\Delta\Phi_j(\tau,\omega) \right)} \frac{d\tau}{\tau} , \\
I_3(t,\omega)&=& -t^2\overline{h_j'(\omega)}\overline{\omega}\fint_{\T}  \frac{\left[\tau \Phi_j'(\tau)-\omega \Phi_j'(\omega)\right] \vert h_j(\tau)-h_j(\omega)\vert^2}{\Big(\Delta\Phi_j(\tau,\omega) \Big) \left( \Delta^t_{\Phi_j}( h_j )(\tau,\omega)\right)\left(\Delta^t_{\Phi_j}( h_j)(\tau,\omega)+\Delta\Phi_j(\tau,\omega) \right)}  \frac{d\tau}{\tau} \\
&&-t \overline{h_j'(\omega)}\overline{\omega}\fint_{\T} \frac{ \left[\tau \Phi_j'(\tau)-\omega \Phi_j'(\omega)\right] \left[2\textnormal{Re}\bigg(\overline{\Big(h_j(\tau)-h_j(\omega)\Big)}\Big(\Phi_j(\tau)-\Phi_j(\omega)\Big)\bigg)\right]}{\Big(\Delta\Phi_j(\tau,\omega) \Big) \left( \Delta^t_{\Phi_j}( h_j )(\tau,\omega)\right)\left(\Delta^t_{\Phi_j}( h_j)(\tau,\omega)+\Delta\Phi_j(\tau,\omega) \right)}  \frac{d\tau}{\tau} 
\end{eqnarray*}
Moreover, using in addition $\ref{diffrac2}$ we can also rewrite $I_1$.
\begin{eqnarray*}
I_1(t,\omega)&=& -t \overline{\Phi_j'(\omega)}\overline{\omega}\fint_{\T}   \frac{\left[\tau \Phi_j'(\tau)-\omega \Phi_j'(\omega)\right]\vert h_j(\tau)-h_j(\omega)\vert^2}{\Big(\Delta\Phi_j(\tau,\omega) \Big) \left( \Delta^t_{\Phi_j}( h_j )(\tau,\omega)\right)\left(\Delta^t_{\Phi_j}( h_j)(\tau,\omega)+\Delta\Phi_j(\tau,\omega) \right)} \frac{d\tau}{\tau} \\
&&-t^2 \overline{\Phi_j'(\omega)}\overline{\omega}\fint_{\T}   \frac{\left[\tau \Phi_j'(\tau)-\omega \Phi_j'(\omega)\right]\left[\textnormal{Re}\bigg(\overline{\Big(h_j(\tau)-h_j(\omega)\Big)}\Big(\Phi_j(\tau)-\Phi_j(\omega)\Big)\bigg)\right] \vert h_j(\tau)-h_j(\omega)\vert^2}{\Big(\Delta\Phi_j(\tau,\omega) \Big)^3 \left(\Delta^t_{\Phi_j}( h_j)(\tau,\omega) \right) \left(\Delta^t_{\Phi_j}( h_j)(\tau,\omega)+\Delta\Phi_j(\tau,\omega) \right)} \frac{d\tau}{\tau} \\
&&-t^2 \overline{\Phi_j'(\omega)}\overline{\omega}\fint_{\T}   \frac{\left[\tau \Phi_j'(\tau)-\omega \Phi_j'(\omega)\right]\left[\textnormal{Re}\bigg(\overline{\Big(h_j(\tau)-h_j(\omega)\Big)}\Big(\Phi_j(\tau)-\Phi_j(\omega)\Big)\bigg)\right]\vert h_j(\tau)-h_j(\omega)\vert^2}{\left(\Delta^t_{\Phi_j}( h_j)(\tau,\omega)\right) \Big(\Delta\Phi_j(\tau,\omega)\Big)^2 \left(\Delta^t_{\Phi_j}( h_j) (\tau,\omega)+ \Delta\Phi_j(\tau,\omega)\right)^2} \frac{d\tau}{\tau} \\ 
&&-2t \overline{\Phi_j'(\omega)}\overline{\omega}\fint_{\T}   \frac{\left[\tau \Phi_j'(\tau)-\omega \Phi_j'(\omega)\right]\left[\textnormal{Re}\bigg(\overline{\Big(h_j(\tau)-h_j(\omega)\Big)}\Big(\Phi_j(\tau)-\Phi_j(\omega)\Big)\bigg)\right]^2}{\Big(\Delta\Phi_j(\tau,\omega) \Big)^3 \left(\Delta^t_{\Phi_j}( h_j)(\tau,\omega) \right) \left( \Delta^t_{\Phi_j}( h_j)(\tau,\omega)+\Delta\Phi_j(\tau,\omega) \right)} \frac{d\tau}{\tau} \\
&&-2t \overline{\Phi_j'(\omega)}\overline{\omega}\fint_{\T}   \frac{\left[\tau \Phi_j'(\tau)-\omega \Phi_j'(\omega)\right]\left[\textnormal{Re}\bigg(\overline{\Big(h_j(\tau)-h_j(\omega)\Big)}\Big(\Phi_j(\tau)-\Phi_j(\omega)\Big)\bigg)\right]^2}{\left(\Delta^t_{\Phi_j}( h_j)(\tau,\omega)\right) \Big(\Delta\Phi_j(\tau,\omega)\big)^2 \left(\Delta^t_{\Phi_j}( h_j) (\tau,\omega)+ \Delta\Phi_j(\tau,\omega)\right)^2} \frac{d\tau}{\tau} 
\end{eqnarray*}  
One may see that we just need to check that the integral term of $I_i(t,\omega)$ belongs to $\tilde{Y}^{k-1}$. We introduce a model integral term, the others term are controled in a similarly way. For any $\omega \in \T$,
\begin{align*}
P(\omega) &\triangleq \fint_{\T}  \frac{\Big(\tau \Phi_j'(\tau)-\omega \Phi_j'(\omega)\Big) \Big( h(\tau)-h(\omega)\Big) \Big( h(\overline{\tau})-h(\overline{\omega}) \Big)}{\left(\Delta^t_{\Phi_j}( h_j) (\tau,\omega)+\Delta\Phi_j(\tau,\omega)\right) \Big(\Delta\Phi_j(\tau,\omega)\Big) \left(\Delta^t_{\Phi_j}( h_j) (\tau,\omega)\right)}\frac{d\tau}{\tau} \\
&= \omega \fint_{\T}  \frac{\Big((\tau-1) \Phi_j'(\tau\omega)+\Phi_j'(\tau\omega)- \Phi_j'(\omega) \Big) \Big( h_j(\tau\omega)-h_j(\omega)\Big)  \Big( h_j(\frac{\overline{\tau}}{\omega})-h_j(\frac{1}{\omega}) \Big)}{\left(\Delta^t_{\Phi_j}( h_j) (\tau \omega,\omega)+\Delta\Phi_j(\tau \omega ,\omega)\right) \Big(\Delta\Phi_j(\tau \omega,\omega)\Big) \left(\Delta^t_{\Phi_j}( h_j) (\tau \omega,\omega)\right)}\frac{d\tau}{\tau}
\end{align*}
Following the same idea of the Lemma $\ref{anal}$, we can write
\[ \frac{1}{\Delta\Phi_j(\tau \omega,\omega) +\Delta^t_{\Phi_j}( h_j)(\tau \omega,\omega)}=\frac{1}{b^{j-1}\vert \tau-1\vert}  \frac{1}{\sqrt{g_j(\tau,\omega) g_j(\overline{\tau},\frac{1}{\omega})}+\sqrt{\tilde{g}_j(\tau,\omega)\tilde{g}_j(\overline{\tau},\frac{1}{\omega})}} .\]
Where $g_j$ and $\tilde{g}_j$ can be extended in the usual ways as follows,
\[\forall z \in \overline{\mathbb{\Delta}_{\varepsilon}},  \; g_j(\tau,z)=1+\frac{f_j(\tau z)-f_j(z)}{b^{j-1}z(\tau-1)} \text{ and } \tilde{g}_j(\tau, z)=g_j(\tau,z)+t \frac{h_j(\tau z)-h_j(z)}{b^{j-1}z(\tau-1)}.\]
As before we can extend P analytically in $C_\varepsilon$ and control the $L^2$ norm of the inner restriction $\omega \in \T \mapsto P(\varepsilon^{\pm1} \omega)$.
We just give few details to control the $L^2$ norm  of the leading term of $\partial_z^{k-1}P(\varepsilon \cdot)$, the proof for the control of $\partial_z^{k-1}P(\varepsilon^{-1} \cdot)$ is similar. Using the  same arguments than before we may write for $z \in \overline{\mathcal{C}_{\varepsilon}}$
\[ \partial_z^{k-1}P(z)=z \fint_{\T}  \frac{\Big(( \partial_z^k \Phi_j )(\tau z)- ( \partial_z^k \Phi_j )(z)\Big) \Big(h_j(\tau z)-h_j(z)\Big) \Big(h_j(\frac{\overline{\tau}}{z})-h_j(\frac{1}{z})\Big) g_j(\tau,z)^{-\frac{1}{2}}g_j(\overline{\tau},\frac{1}{z})^{-\frac{1}{2}}}{b^{3(j-1)}\vert \tau-1\vert^3 \Big( \sqrt{g_j(\tau,z) g_j(\overline{\tau},\frac{1}{z})}+\sqrt{\tilde{g}_j(\tau,z)\tilde{g}_j(\overline{\tau},\frac{1}{z})}\Big) \sqrt{\tilde{g}_j(\tau,z) \tilde{g}_j(\overline{\tau},\frac{1}{z})}} \frac{d\tau}{\tau}+ l.o.t\]
Applying the Lemma $\ref{exp}$ with $\tilde{V_1}=f_j$ and $\tilde{V_2}=f_j+th_j$ we can etablish for $z \in \varepsilon \T \cup \varepsilon^{-1} \T$ the following identity
\begin{align*}
&\frac{ \Big(\sqrt{g(\tau,z) g(\overline{\tau},\frac{1}{z})}+\sqrt{\tilde{g}(\tau,z)\tilde{g}(\overline{\tau},\frac{1}{z})}\Big)^{-1} }{\vert \tau-1\vert}=\frac{(\tau-1)}{\vert \tau-1 \vert}H(\tau,z)\\
&+\frac{\bigg( \sqrt{(1+\frac{(\partial_z f_j)(z)}{b^{j-1}})(1+\frac{(\partial_z f_j)(\frac{1}{z})}{b^{j-1}})}+\sqrt{(1+\frac{\big(\partial_z(f_j+th_j)\big)(z)}{b^{j-1}}\big) \big(1+\frac{\big(\partial_z (f_j+th_j) \big)(\frac{1}{z})}{b^{j-1}}\big)}\bigg)^{-1}}{\vert \tau-1 \vert}.
\end{align*}
With $H(\cdot ,\varepsilon^{\pm} \cdot)\in L^{\infty}(\T \times \T)$.\\
Consequently, this identity allows us to deal with the leading term of $\partial_z^{k-1}P$. It follows for $l \in \lbrace1,2,3 \rbrace $
\[ \underset{t \rightarrow 0}{\textnormal{lim}} \Vert I_l\Vert_{ \tilde{Y}^{k-1}} =0. \]
Eventually, we have proved than $F_1$  can be extended analytically on $C_{\varepsilon}$ and 
\[ \underset{t \rightarrow 0}{\textnormal{lim}} \Vert F_1\Vert_{ Y^{k-1}} =0. \]
Moreover, the interaction term $F_2$ is dealed with the same arguments than the regularity of the interaction term.
Our next task is to prove that
\[DG_j: \R\times V_r \times V_r \rightarrow \mathcal{L}(\R\times  X^{k+\log} \times X^{k+\log} , Y^{k-1})\]
is  well-defined and continuous.\\
For the first part, the non trivial point  is that $\forall (\Omega, f_1,f_2) \in \R\times V_r \times V_r $, $DG_j(\Omega, f_1,f_2) \in \mathcal{L}(\R\times  X^{k+\log} \times X^{k+\log} , Y^{k-1})$. The linearity is obvious.\\
As before, we  just give details about the continuity of the self-induced term $DO_j(\Omega,f_j)$ . To begin we rewrite
\[DO_j(\Omega,f_j)(h_j)(\omega)=\sum_{i=1}^4 \frac{\hat{I}_i(\omega)-\hat{I}_i(\frac{1}{\omega})}{2i}\]
with 
\begin{align*}
\hat{I}_1(\omega)&=\Omega \frac{1}{\omega} \left[ \Phi_j(\omega)h_j'(\frac{1}{\omega})+h_j(\omega)\Phi_j'(\frac{1}{\omega})\right],\\
\hat{I}_2(\omega)&=(-1)^j h_j'(\frac{1}{\omega})\frac{1}{\omega} \fint_{\T}  \frac{\tau\Phi_j'(\tau)-\omega \Phi_j'(\omega)}{\Delta\Phi_j(\tau ,\omega) }\frac{d\tau}{\tau}, \\
\hat{I}_3(\omega)&=(-1)^j \Phi_j'(\frac{1}{\omega})\frac{1}{\omega} \fint_{\T}  \frac{\tau h_j'(\tau)-\omega h_j'(\omega)}{\Delta\Phi_j(\tau ,\omega) }\frac{d\tau}{\tau}, 
\end{align*}
and
\[\hat{I}_4(\omega)=-(-1)^j \Phi_j'(\frac{1}{\omega})\frac{1}{\omega} \fint_{\T}  \frac{(\tau\Phi_j'(\tau)-\omega \Phi_j'(\omega))\textnormal{Re}\bigg(\Big(\overline{h_j(\tau)-h_j(\omega)}\Big)\Big(\Phi_j(\tau)-\Phi_j(\omega)\Big)\bigg)}{\big( \Delta\Phi_j(\tau,\omega)\big)^3} \frac{d\tau}{\tau}. \]
Using the Lemma $ \ref{anal}$ and an adaptation, one may find a constant C such that for $p \in \lbrace 1,2,3,4 \rbrace$ the following estimate is checked 
\[ \Vert \hat{I}_p \Vert_{\tilde{Y}^{k-1}} \leq C \Vert \Phi_j \Vert_{X^{k+\log}} \Vert h_j \Vert_{X^{k+\log}}. \]
 Consequently, $DG_j$ is well-defined.
The continuity of $DG_j$ is the final point of the proof. We just explain the  continuity of $DO_j(\Omega,\cdot)$.
Let be $f_j,\tilde{f_j}\in V_r \times V_r$ and $h_j\in X^{k+\log}$ with $\Vert h_j \Vert_{X^{k+\log}} =1$, we have for any $\omega \in \T$
\[ DO_j(\Omega,f_j)(h_j)(\omega)-DO_j(\Omega,\tilde{f_j})(h_j)(\omega)= \sum_{i=1}^9 \frac{\tilde{I}_i(\omega)-\tilde{I}_i(\frac{1}{\omega})}{2i}\]
with
\begin{align*}
\tilde{I}_1(\omega)&= \frac{\Omega}{\omega} \Big(\Phi_j(\omega)-\tilde{\Phi}_j(\omega) \Big) h_j'(\frac{1}{\omega})+ \frac{\Omega}{\omega} \Big(\Phi_j'(\frac{1}{\omega})-\tilde{\Phi}_j'(\frac{1}{\omega}) \Big)h_j(\omega),\\
\tilde{I}_2(\omega)& =\frac{(-1)^{j}}{\omega} h_j'(\frac{1}{\omega})\fint_{\T} \frac{\tau \Big(\Phi_j'(\tau)-\tilde{\Phi}_j'(\tau) \Big)-\omega \Big(\Phi_j'(\omega)-\tilde{\Phi}_j'(\omega) \Big)}{\Delta\Phi_j(\tau ,\omega) } \frac{d\tau}{\tau},\\
 \tilde{I}_3(\omega) &=\frac{(-1)^{j}}{\omega} h_j'(\frac{1}{\omega})\fint_{\T} \left[\tau\tilde{\Phi}_j'(\tau)-\omega \tilde{\Phi}_j'(\omega) \right] \left[\frac{1}{\Delta\Phi_j(\tau ,\omega)} -\frac{1}{\Delta\tilde{\Phi}_j(\tau ,\omega)}\right]\frac{d\tau}{\tau} ,\\
\tilde{I}_4(\omega)& =\frac{(-1)^{j}}{\omega} \Big(\Phi_j'(\frac{1}{\omega})-\tilde{\Phi}_j'(\frac{1}{\omega}) \Big)\fint_{\T} \frac{\tau h_j'(\tau)-\omega h_j'(\omega)}{\Delta\Phi_j(\tau ,\omega) } \frac{d\tau}{\tau}, \\
\tilde{I}_5(\omega) &=\frac{(-1)^{j}}{\omega} \tilde{\Phi}_j'(\frac{1}{\omega})\fint_{\T} \Big(\tau h_j'(\tau)-\omega h_j'(\omega) \Big) \left[\frac{1}{\Delta\Phi_j(\tau,\omega) } -\frac{1}{\Delta\tilde{\Phi}_j(\tau ,\omega) }\right]\frac{d\tau}{\tau},  
\end{align*}
\begin{align*}
\tilde{I}_6(\omega) &=\frac{(-1)^{j}}{\omega} \Big(\tilde{\Phi}_j'(\frac{1}{\omega})-\Phi_j'(\frac{1}{\omega}) \Big)\fint_{\T} \frac{\bigg( \tau \tilde{\Phi}_j'(\tau)-\omega \tilde{\Phi}_j'(\omega) \bigg)\textnormal{Re}\bigg( \Big(\overline{h_j(\omega)-h_j(\tau)} \Big) \Big(\tilde{\Phi}_j(\omega)-\tilde{\Phi}_j(\tau) \Big) \bigg)}{\big( \Delta\tilde{\Phi}_j(\tau ,\omega) \big)^3 }\frac{d\tau}{\tau}, \\
\tilde{I}_7(\omega)& =\frac{(-1)^{j}}{\omega} \Phi_j'(\frac{1}{\omega})\fint_{\T} \frac{\bigg( \tau \Big(\tilde{\Phi}_j'(\tau)-\Phi_j'(\tau) \Big)-\omega \Big(\tilde{\Phi}_j'(\omega)-\Phi_j'(\omega) \Big) \bigg) \textnormal{Re} \bigg( \Big(\overline{h_j(\omega)-h_j(\tau)} \Big) \Big(\tilde{\Phi}_j(\omega)-\tilde{\Phi}_j(\tau)\Big) \bigg)}{ \big(\Delta\tilde{\Phi}_j(\tau ,\omega) \big)^3 }\frac{d\tau}{\tau}, \\
\tilde{I}_8(\omega)& =\frac{(-1)^{j}}{\omega} \Phi_j'(\frac{1}{\omega})\fint_{\T} \frac{\Big(\tau \Phi_j'(\tau)-\omega \Phi_j'(\omega)\Big)\textnormal{Re} \bigg( \Big(\overline{h_j(\omega)-h_j(\tau)} \Big)\Big( \big(\tilde{\Phi}_j(\omega)-\Phi_j(\omega) \big) -\big(\tilde{\Phi}_j(\tau)-\Phi_j(\tau)\big)\Big)\bigg)}{\big( \Delta\tilde{\Phi}_j(\tau ,\omega) \big)^3 }\frac{d\tau}{\tau} 
\end{align*}
and
\begin{align*}
\tilde{I}_9(\omega) &=\frac{(-1)^{j}}{\omega} \Phi_j'(\frac{1}{\omega})\fint_{\T} \big(\tau\Phi_j'(\tau)-\omega\Phi_j'(\omega) \big) \textnormal{Re}\bigg( \Big( \overline{h_j(\omega)-h_j(\tau)} \Big) \Big(\Phi_j(\omega)-\Phi_j(\tau) \Big) \bigg)\times\\
&\left[ \frac{1}{ \big( \Delta\tilde{\Phi}_j(\tau ,\omega) \big)^3 }-\frac{1}{ \big( \Delta\Phi_j(\tau ,\omega) \big)^3 }\right] \frac{d\tau}{\tau} .
\end{align*}
 For $p\in \lbrace 1,2,4,6,7,8 \rbrace$, one may extend $\tilde{I}_p$ as before. The control of the $\tilde{Y}^{k-1}$ norm leads on the lemma $\ref{anal}$ and an adaptation, we can find a constant C such that 
\[ \Vert \tilde{I}_p \Vert_{\tilde{Y}^{k-1}} \leq C \Vert \Phi_j-\tilde{\Phi}_j \Vert_{X^{k+\log}}.\]
We give few details for the integral term of $\tilde{I}_3$. As
\begin{align*}
 \frac{1}{\Delta \Phi_j(\tau ,\omega)}-\frac{1}{\Delta\tilde{\Phi}_j(\tau ,\omega)}&=\frac{\bigg(\tilde{\Phi}_j(\tau)-\tilde{\Phi}_j(\omega)\bigg)\bigg( \Big(\tilde{\Phi}_j(\overline{\tau})-\Phi_j(\overline{\tau})\Big)- \Big(\tilde{\Phi}_j(\frac{1}{\omega})-\Phi_j(\frac{1}{\omega})\Big) \bigg)}{\big( \Delta \Phi_j(\tau ,\omega) \big) \big( \Delta\tilde{\Phi}_j(\tau ,\omega)\big)  \big(\Delta \Phi_j(\tau ,\omega)+\Delta\tilde{\Phi}_j(\tau ,\omega) \big)} \\
 &+\frac{\bigg(\Phi_j(\overline{\tau})-\Phi_j(\frac{1}{\omega})\bigg)   \bigg( \tilde{\Phi}_j(\tau)-\Phi_j(\tau)-\tilde{\Phi}_j(\omega)+\Phi_j(\omega)\bigg)}{\big(\Delta \Phi_j(\tau ,\omega)\big) \big( \Delta\tilde{\Phi}_j(\tau ,\omega)\big) \big(\Delta\Phi_j(\tau ,\omega)+\Delta\tilde{\Phi}_j(\tau ,\omega)\big)}.
\end{align*}
The integral can be split in two terms and we just give few details for one. We deal with the other in the same way. After a change of variabe, we shall extend and control the term
\[\tilde{\tilde{I}}_3(\omega) \triangleq   \omega \fint_{\T}\frac{ \bigg(\tau\tilde{\Phi}_j'(\tau \omega)- \tilde{\Phi}_j'(\omega)\bigg) \bigg( \tilde{\Phi}_j(\tau\omega)-\tilde{\Phi}_j(\omega) \bigg) \bigg( \Big(\tilde{\Phi}_j(\frac{\overline{\tau}}{\omega})-\Phi_j(\frac{\overline{\tau}}{\omega})\Big)-\Big(\tilde{\Phi}_j(\frac{1}{\omega})-\Phi_j(\frac{1}{\omega})  \Big) \bigg)}{\left(\Delta\Phi_j(\tau \omega,\omega)\right) \left(\Delta\tilde{\Phi}_j(\tau \omega,\omega)\right) \left(\Delta\Phi_j(\tau \omega,\omega)+\Delta\tilde{\Phi}_j(\tau \omega,\omega) \right)} \frac{d\tau}{\tau}.\]
As before, we can write
\[ \Delta\Phi_j(\tau \omega,\omega)=b^{j-1}\vert \tau -1 \vert\sqrt{g(\tau,\omega)g(\overline{\tau},\frac{1}{\omega})},\]
\[ \Delta\tilde{\Phi}_j(\tau \omega,\omega)=b^{j-1} \vert \tau -1 \vert \sqrt{\tilde{g}(\tau,\omega)\tilde{g}(\overline{\tau},\frac{1}{\omega})},\]
and
\[ \Delta\Phi_j(\tau \omega,\omega)+\Delta\tilde{\Phi}_j(\tau \omega,\omega) =b^{j-1}\vert \tau-1\vert \Big( \sqrt{g(\tau,\omega)g(\overline{\tau},\frac{1}{\omega})}+ \sqrt{\tilde{g}(\tau,\omega)\tilde{g}(\overline{\tau},\frac{1}{\omega})} \Big).\]
where we can extend $g$ and $\tilde{g}$ as usual,
\[\forall z \in \overline{\mathbb{\Delta_{\varepsilon}}}, \; g(\tau, z)=1+\frac{f_j( \tau z)-f_j(z)}{b^{j-1} \omega (\tau-1)} \text{ and }\tilde{g}(\tau,z)=1+\frac{\tilde{f}_j( \tau z)-\tilde{f}_j(z)}{b^{j-1} \omega (\tau-1)}.\]
Thus the holomorpic extension of $\tilde{\tilde{I}}_3$ on $C_{\varepsilon}$  is given by
\[\tilde{\tilde{I}}_3(z) =z \fint_{\T}\frac{ \Big( \tau \partial_z \tilde{\Phi}_j(\tau z) - \partial \tilde{\Phi}_j(z) \Big) \Big(\tilde{\Phi}_j(\tau z)-\tilde{\Phi}_j(z) \Big) \Big(\tilde{\Phi}_j(\frac{\overline{\tau}}{z})-\Phi_j(\frac{\overline{\tau}}{z})-\tilde{\Phi}_j(\frac{1}{z})+\Phi_j(\frac{1}{z}) \Big)}{b^{3(j-1)}\vert \tau-1 \vert^3  \sqrt{\tilde{g}(\tau,z)\tilde{g}(\overline{\tau},\frac{1}{z})}\sqrt{g(\tau,z)g(\overline{\tau},\frac{1}{z})} \Big( \sqrt{g(\tau,z)g(\overline{\tau},\frac{1}{z})}+ \sqrt{\tilde{g}(\tau,z)\tilde{g}(\overline{\tau},\frac{1}{z})}\Big)} \frac{d\tau}{\tau}.\]
Concerning the $\tilde{Y}^{k-1}$ norm, we just give some details for the $L^2$ norm of the leading term of $\partial_z^{k-1}\tilde{\tilde{I}}_3(\varepsilon \cdot)$. For $z \in \overline{\mathcal{C}_{\varepsilon}}$, one may write
\begin{align*}
\partial_z^{k-1} \tilde{\tilde{I}}_3(z)= \fint_{\T}\frac{ \Big( \partial_z^k \tilde{\Phi}_j(\tau z)-\partial_z^k \tilde{\Phi}_j(z)\Big) \Big(\tilde{\Phi}_j( \tau z)-\tilde{\Phi}_j(z) \Big) \Big(\tilde{\Phi}_j(\frac{\overline{\tau}}{z})-\Phi_j(\frac{\overline{\tau}}{z})-\tilde{\Phi}_j(\frac{1}{z})+\Phi_j(\frac{1}{z}) \Big)}{\vert \tau-1 \vert^3  \sqrt{\tilde{g}(\tau,z)\tilde{g}(\overline{\tau},\frac{1}{z})} \sqrt{g(\tau,z)g(\overline{\tau},\frac{1}{z})}\Big( \sqrt{g(\tau,z)g(\overline{\tau},\frac{1}{z})}+ \sqrt{\tilde{g}(\tau,z) \tilde{g}(\overline{\tau},\frac{1}{z})}
\Big)} \frac{d\tau}{\tau} +l.o.t
\end{align*}
The control of the $L^2$ norm of the inner restriction $(\omega \mapsto \partial_z^{k-1} \tilde{\tilde{I}}_3(\varepsilon \omega))$ must ensure the continuity. First, one may obtain the following identity for $z \in \varepsilon \T \cup \varepsilon^{-1} \T$:
\[\frac{\Big(\tilde{\Phi}_j(\tau  z)-\tilde{\Phi}_j(z)\Big) \Big(\tilde{\Phi}_j(\frac{\overline{\tau}}{z})-\Phi_j(\frac{\overline{\tau}}{z})-\tilde{\Phi}_j(\frac{1}{z})+\Phi_j(\frac{1}{z}) \Big)}{\vert \tau-1 \vert^3  \sqrt{\tilde{g}(\tau,z)\tilde{g}(\overline{\tau},\frac{1}{z})}\sqrt{g(\tau,z)g(\overline{\tau},\frac{1}{z})}\Big( \sqrt{g(\tau,z)g(\overline{\tau},\frac{1}{z})}+ \sqrt{\tilde{g}(\tau,z)\tilde{g}(\overline{\tau},\frac{1}{z})} \Big)}=\frac{K(z)}{\vert \tau-1 \vert}+\frac{(\tau-1)}{\vert \tau-1 \vert }H(\tau,z)\]
with the  estimations
\begin{equation}\label{ker}
 \Vert K \Vert_{L^{\infty}(\varepsilon^{\pm} \T)} \leq C \Vert \Phi_j-\tilde{\Phi}_j \Vert_{X^{k+\log}} \text{, }\Vert H \Vert_{L^{ \infty}(\T \times \varepsilon^{\pm} \T)} \leq C \Vert \Phi_j-\tilde{\Phi}_j \Vert_{X^{k+\log}}.
 \end{equation}
The proof leads on  the lemma $\ref{exp}$ and an adaptation.
 Hence, we can estimate for  $p \in \lbrace 3,5 \rbrace$:
 \[\Vert \tilde{I}_p \Vert_{\tilde{Y}^{k-1}} \leq C \Vert \Phi_j-\tilde{\Phi}_j \Vert_{X^{k+\log}}\]
For $\tilde{I}_9$ we just need to notice this decomposition
\begin{align*}
&\frac{1}{\big( \Delta \Phi_j(\tau ,\omega) \big)^3}-\frac{1}{\big(  \Delta\Phi_j(\tau ,\omega) \big)^3} =\frac{1}{\big( \Delta\Phi_j(\tau ,\omega) \big)^2}\left[ \frac{1}{ \Delta\Phi_j(\tau ,\omega) }-\frac{1}{ \Delta\tilde{\Phi}_j(\tau ,\omega) } \right]\\
&+\frac{1}{\big( \Delta\Phi_j(\tau ,\omega) \big) \big( \Delta \tilde{\Phi}_j(\tau ,\omega) \big) }\left[ \frac{1}{\Delta\Phi_j(\tau ,\omega) }-\frac{1}{ \Delta \tilde{\Phi}_j(\tau ,\omega) } \right]\\
& +\frac{1}{\big( \Delta \tilde{\Phi}_j(\tau ,\omega) \big)^2 } \left[ \frac{1}{\Delta\Phi_j(\tau ,\omega) }-\frac{1}{\Delta \tilde{\Phi}_j(\tau ,\omega) } \right] .
\end{align*}
With this writting and the same arguments than for $\tilde{I}_3$ , we get
 \[\Vert \tilde{I}_9 \Vert_{\tilde{Y}^{k-1}} \leq C \Vert \Phi_j-\tilde{\Phi}_j \Vert_{X^{k+\log}}.\]
Finally, $DG_j$ is continuous.
\end{proof}
\section{Study of the linearized operator}\label{Linea12}
The main task of this section is to perform a spectral study of the linearized operator of the functional $G$ introduced in \eqref{V-state} at the annular solution $(\hbox{Id},b\hbox{Id})$. The first subsection is dedicated to an explicit computation of this operator and to get a more user-friendly expression through some basic identities on hypergeometric functions.  In the second part, we want to find the values of $\Omega$ leading to a one-dimensional kernel for the linearized operator.  We show that for each frequency mode this study reduces to  a second degree equation on the variable $\Omega$. The dimension of the kernel is achieved through the strict monotonicity  of the eigenvalues with respect to the frequency.  Lastly, we check the  full assumptions of the Crandall-Rabinowitz's  theorem especially the transversality condition which holds only when the eigenvalues are simple.
\subsection{Linearized operator}\label{Linearized operator}
The primary purpose of this section is to compute the linearized operator of $G$ at the trivial solution $(\hbox{Id}, b\hbox{Id})$ and to reach a more simplified and compact expression.
Since $G=(G_1,G_2)$ then for given $(h_1,h_2) \in X^{k+\log} \times X^{k+\log} $, we have
\[ DG(\Omega,0,0)(h_1,h_2)=\left( \begin{array}{c}
D_{f_1}G_1(\Omega,0,0)h_1+D_{f_2}G_1(\Omega,0,0)h_2 \\ 
D_{f_1}G_2(\Omega,0,0)h_1+D_{f_2}G_2(\Omega,0,0)h_2
\end{array} \right).\]
Replacing in $(\ref{diff})$, $(\ref{lj})$ and $(\ref{Nj})$ $\Phi_1$ by $\hbox{Id}$ and $\Phi_2$ by $b \hbox{Id}$  yields
\begin{equation*}\label{DG1}
DG_1(\Omega,0,0)(h_1,h_2)(\omega)=\Omega\, \mathcal{L}_0(h_1)(\omega)+\mathcal{L}_1(h_1)(\omega)+\mathcal{L}_2(h_1,h_2)(\omega),
\end{equation*}
\begin{equation*}\label{DG2}
DG_2(\Omega,0,0)(h_1,h_2)(\omega)=\Omega \,b \mathcal{L}_0(h_2)(\omega)+\mathcal{L}_1(h_2)(\omega)+\mathcal{L}_3(h_1,h_2)(\omega)
\end{equation*}
with
 \[ \mathcal{L}_0(h_j)(\omega)=\textnormal{Im}\left\lbrace h_j'(\overline{\omega})+h_j(\omega)\overline{\omega}\right\rbrace,\]
\begin{align*}  
  \mathcal{L}_1(h_j)(\omega)&=\textnormal{Im}\left\lbrace (-1)^j\overline{h_j'(\omega)}\overline{\omega} \fint_{\T}  \frac{\tau-\omega }{\vert \omega-\tau\vert }\frac{d\tau}{\tau}-(-1)^j  \overline{\omega}\fint_{\T}  \frac{(\tau-\omega)\textnormal{Re}\Big((\overline{h_j(\tau)-h_j(\omega)})(\tau-\omega)\Big)}{\vert \tau-\omega \vert^3} \frac{d\tau}{\tau}\right\rbrace\\
&+\textnormal{Im}\left\lbrace(-1)^j\overline{\omega}\fint_{\T}  \frac{\tau h_j'(\tau)-\omega h_j'(\omega)}{\vert \omega-\tau\vert }\frac{d\tau}{\tau} \right\rbrace,
\end{align*}
\begin{eqnarray*}
\mathcal{L}_2(h_1,h_2)(\omega)&=& \textnormal{Im}\left\lbrace\overline{h_1'(\omega)}\overline{\omega}\fint_{\T}  \frac{b\tau-\omega }{\vert b\tau-\omega\vert }\frac{d\tau}{\tau} \right\rbrace-\textnormal{Im} \left\lbrace \overline{\omega}\fint_{\T}  \frac{ \omega h_1'(\omega)-\tau h_2'(\tau)}{\vert \omega-b\tau\vert }\frac{d\tau}{\tau} \right\rbrace\\
&+&\textnormal{Im} \left\lbrace-\overline{\omega}\fint_{\T}  \frac{(b\tau-\omega )\textnormal{Re}\Big((h_1(\omega)-h_2(\tau))\overline{(\omega-b\tau)}\Big)}{\vert b\tau-\omega \vert^3} \frac{d\tau}{\tau} 
\right\rbrace
\end{eqnarray*}
and
\begin{align*}
\mathcal{L}_3(h_1,h_2)(\omega)&= \textnormal{Im}\left(-\overline{\omega}h_2'(\overline{\omega})\fint_{\T}  \frac{\tau-b\omega }{\vert \tau-b\omega\vert }\frac{d\tau}{\tau}+\overline{\omega}b\fint_{\T}  \frac{\omega h_2'(\omega)-\tau h_1'(\tau)}{\vert b\omega-\tau\vert }\frac{d\tau}{\tau}\right) \\
&+\textnormal{Im} \left(b\,\overline{\omega}\fint_{\T}  \frac{(\tau-b\omega )\textnormal{Re}\Big((h_1(\tau)-h_2(\omega))(\overline{\tau-b\omega})\big)}{\vert \tau-b\omega \vert^3} \frac{d\tau}{\tau} \right).
\end{align*}
We shall now compute the Fourier series of the mapping $\omega \mapsto DG(\Omega,0,0)(h_1,h_2)(\omega)$ with
\[h_1(\omega) =\sum_{n=1}^{+ \infty} a_n\overline{\omega}^n \text{ and } h_2(\omega) =\sum_{n=1}^{+ \infty} c_n\overline{\omega}^n \text{, } \omega \in \T \]
where $a_n$ and $c_n$ are real  for all the values $ n\in \N^{\star}$. This is summarized in the following proposition.
\begin{proposition}\label{Matrice}
Let $b\in (0,1)$, $n\in \N^{\star}$ and define 
\[ \Lambda_n(b) \triangleq \frac{1}{b} \int_0^{+ \infty} J_n(bt)J_n(t) dt,\quad S_n \triangleq \frac{2}{\pi}\sum_{k=1}^{n-1}\frac{1}{2k+1}\]
with  $J_n$ refers to the Bessel function of the first kind. Then, we have
 \[DG(\Omega,0,0)(h_1,h_2)(\omega)=\frac{i}{2}\sum_{n\geq 1}(n+1)M_{n+1} \left( \begin{array}{c}
a_n \\ 
c_n
\end{array} \right) \left(\omega^{n+1}-\overline{\omega}^{n+1}\right) \text{, } \forall \omega \in \T \]
where the matrix $M_n$ is given for $n\geq 2$ by:
\[M_n=\left(\begin{array}{cc}
\Omega-S_n+b^2\Lambda_1(b) & -b^2\Lambda_n(b) \\ 
b\Lambda_n(b) & b\Omega+S_n-b\Lambda_1(b)
\end{array}\right). \]
\end{proposition}
\begin{proof}
 We begin with the  easier  term  $ \mathcal{L}_0(h_j)(\omega)$. Thus by straightforward computations we obtain
\[\mathcal{L}_0(h_1)(\omega)=\frac{i}{2}\sum_{n=1}^{+ \infty}(n+1) a_{n}(\omega^{n+1}-\overline{\omega}^{n+1}).\]
and
\[\mathcal{L}_0(h_2)(\omega)=\frac{i}{2}\sum_{n=1}^{+ \infty}(n+1) c_{n}(\omega^{n+1}-\overline{\omega}^{n+1}).\]
The computation of  $\mathcal{L}_1(h_1)(\omega)$  lies on the following identities whose proofs can be found in  \cite{15}.
Let  $n\in \N^\star$ and  $\omega \in \T$ then
\begin{equation}\label{c1}
 \fint_{\T} \frac{\tau^n-\omega^n}{\vert \omega- \tau\vert }\frac{d\tau}{\tau}=-\frac{2\omega^n}{\pi}\sum_{k=0}^{n-1}\frac{1}{2k+1}
 \end{equation}
and
 \begin{equation}\label{c2}
  \fint_{\T} \frac{(\tau-\omega)^2(\tau^n-\omega^n)}{\vert \omega- \tau\vert^3 }\frac{d\tau}{\tau}=\frac{2\omega^{n+2}}{\pi}\sum_{k=1}^{n}\frac{1}{2k+1}.
  \end{equation}
Performing straightforward computations we obtain the result
\begin{align*}
\mathcal{L}_1(h_1)(\omega)&=\textnormal{Im} \left\lbrace \sum_{n=1}^{+ \infty} na_{n}\omega^{n}\fint_{\T}\frac{\tau-\omega}{\vert \tau - \omega \vert}\frac{d\tau}{\tau}+ \frac{1}{2\omega}\sum_{n=1}^{+ \infty} a_{n} \fint_{\T} \frac{(\tau^{n}-\omega^{n})(\tau - \omega)^2}{\vert \tau - \omega \vert^3} \frac{d\tau}{\tau} \right\rbrace\\
&+\textnormal{Im} \left\lbrace \overline{\omega}\sum_{n=1}^{+ \infty} n a_{n}\fint_{\T} \frac{\overline{\tau}^{n}-\overline{\omega}^{n}}{\vert \tau - \omega \vert}\frac{d\tau}{\tau}+\frac{1}{2\omega}\sum_{n=1}^{+ \infty} a_{n}\fint_{\T}\frac{\overline{\tau}^{n}-\overline{\omega}^{n}}{\vert \tau - \omega \vert}\frac{d\tau}{\tau} \right\rbrace .
\end{align*}
Noticing the following equality
\[ \forall \tau \in \T, \forall \omega \in \T,  \fint_{\T} \frac{\overline{\tau}^n-\overline{\omega}^n}{\vert \omega- \tau\vert }\frac{d\tau}{\tau}=\overline{ \fint_{\T} \frac{\tau^n-\omega^n}{\vert \omega- \tau\vert }\frac{d\tau}{\tau}},\]
we obtain thanks to $(\ref{c1})$ and $(\ref{c2})$ the following identity
\[\mathcal{L}_1(h_1)(\omega)=-\sum_{n=1}^{+ \infty} \alpha_{n}a_n \textnormal{Im}\left\lbrace \overline{\omega}^{n+1}\right\rbrace +\sum_{n=1}^{+ \infty} \beta_{n}a_n\textnormal{Im}\left\lbrace\omega^{n+1} \right\rbrace\]
where
\[\alpha_{n}\triangleq \frac{2n+1}{\pi}\sum_{k=0}^{n-1}\frac{1}{2k+1}\]
and
\[\beta_{n}\triangleq\ -\frac{2n}{\pi}+\frac{1}{\pi}\sum_{k=1}^{n}\frac{1}{2k+1}.\]
As
\[ \alpha_{n}+\beta_{n}=\frac{2(n+1)}{\pi}\sum_{k=1}^{n}\frac{1}{2k+1}.\]
Finally we get
\[\mathcal{L}_1(h_1)(\omega)=-\frac{i}{2}\sum_{n=1}^{+ \infty} a_{n}(n+1)S_{n+1}(\omega^{n+1} -\overline{\omega}^{n+1}).\]
In the same way we obtain
\[\mathcal{L}_1(h_2)(\omega)=\frac{i}{2}\sum_{n=1}^{+ \infty} c_{n}(n+1)S_{n+1}(\omega^{n+1} -\overline{\omega}^{n+1}).\]
To compute $\mathcal{L}_2(h_1,h_2)(\omega)$ we begin to rewrite
\begin{align*}
-\overline{\omega}\fint_{\T}  \frac{(b\tau-\omega )\textnormal{Re}\Big(h_1(\omega)-h_2(\tau))\overline{(\omega-b\tau)}\Big)}{\vert b\tau-\omega \vert^3} \frac{d\tau}{\tau} &=\frac{\overline{\omega}}{2}\fint_{\T}  \frac{h_1(\omega)-h_2(\tau)}{\vert b\tau-\omega \vert} \frac{d\tau}{\tau} \\
&+\frac{\overline{\omega}}{2}\fint_{\T}  \frac{(b\tau-\omega )^2 \overline{(h_1(\omega)-h_2(\tau))}}{\vert b\tau-\omega \vert^3} \frac{d\tau}{\tau} .
\end{align*} 
Replacing $h_j$ and $h_j'$ by their expressions, we obtain the following identity
\begin{eqnarray*}
\mathcal{L}_2(h_1,h_2)(\omega)&=&\textnormal{Im}\left\lbrace\ -\sum_{n=1}^{+ \infty} na_{n}\fint_{\T}  \frac{b\tau-\omega }{\vert b\tau-\omega\vert }\frac{d\tau}{\tau} +\sum_{n=1}^{+ \infty} n \overline{\omega} \fint_{\T}  \frac{ a_n \overline{\omega}^n- c_n \overline{\tau}^n}{\vert \omega-b\tau\vert }\frac{d\tau}{\tau} +\frac{1}{2}\sum_{n=1}^{+ \infty}\overline{\omega}\fint_{\T}  \frac{a_n \overline{\omega}^n - \overline{\tau}^n}{\vert b\tau-\omega \vert} \frac{d\tau}{\tau} \right\rbrace\\
&+&\textnormal{Im} \left\lbrace \frac{b}{2}\sum_{n=1}^{+ \infty} \overline{\omega} \frac{(b \omega- \tau)(a_n \omega^n - c_n \tau^n)}{\vert b\tau-\omega \vert^3} d\tau -\frac{1}{2}\sum_{n=1}^{+ \infty} \frac{(b - \omega \overline{\tau})(a_n \omega^n - c_n \tau^n)}{\vert b\tau-\omega \vert^3} d \tau \right\rbrace.
\end{eqnarray*}
To compute these terms, we will use the identities proved in  \cite{16}:
Let $b\in (0,1)$ and $n\in \N$, then for any $\omega \in \T$ we have
\begin{equation}\label{c3}
 \fint_{\T} \frac{\tau^{n-1}}{\vert b\tau - \omega \vert} d \tau=\omega^n b^n \frac{(\frac{1}{2})_n}{n!}F\left(\frac{1}{2},n+\frac{1}{2},n+1;b^2\right),
 \end{equation}
\begin{eqnarray}\label{c4}
\nonumber \fint_{\T} \frac{\overline{\tau}^{n+1}}{\vert b\tau - \omega \vert}  d\tau  &=&\overline{\fint_{\T} \frac{\tau^{n-1}}{\vert b\tau - \omega \vert} d\tau} \\
 &=&\overline{\omega}^n b^n \frac{(\frac{1}{2})_n}{n!}F\left(\frac{1}{2},n+\frac{1}{2},n+1;b^2\right),
\end{eqnarray} 
\begin{equation}\label{c5}
\fint_{\T}\frac{\overline{\tau}^{n+1}}{\vert \omega-b\tau \vert^3} d \tau=\omega^n b^n \frac{\left( \frac{3}{2}\right)_n}{n!}F \left( \frac{3}{2},n+ \frac{3}{2}, n+1;b^2\right),
\end{equation}
\begin{equation}\label{c6}
\fint_{\T}\frac{\tau^{n-1}}{\vert \omega-b\tau \vert^3} d \tau=\omega^n b^n \frac{\left( \frac{3}{2}\right)_n}{n!}F \left( \frac{3}{2},n+ \frac{3}{2}, n+1;b^2\right),
\end{equation}
\begin{equation}\label{c7}
\fint_{\T} \frac{(b \tau-\omega)(a \omega^n- c \tau^n)}{\vert \omega-b\tau \vert^3} d \tau= -\omega^{n+2} b \Bigg[ \frac{3}{2}a F\left( \frac{1}{2},\frac{5}{2},2; b^2 \right)-c b^n \frac{\left( \frac{3}{2}\right)_{n+1}}{(n+1)!} F \left( \frac{1}{2}, n+ \frac{5}{2},n+2; b^2 \right) \Bigg],
\end{equation}
and
\begin{equation}\label{c8}
\fint_{\T} \frac{(b \omega-\tau)(c \omega^n- a \tau^n)}{\vert \omega-b\tau \vert^3} d \tau= -\omega^{n+2} b^2 \Bigg[ \frac{3}{8}c F\left( \frac{3}{2},\frac{5}{2},3; b^2 \right)-a b^n \frac{\left( \frac{1}{2}\right)_{n+2}}{(n+2)!} F \left( \frac{3}{2}, n+ \frac{5}{2},n+3; b^2 \right) \Bigg].
\end{equation}
 We shall split the computation in many parts. By using $(\ref{c3})$ and $(\ref{c4})$ we find
\[-\sum_{n=1}^{+ \infty} na_{n}\fint_{\T}  \frac{b\tau-\omega }{\vert b\tau-\omega\vert }\frac{d\tau}{\tau}=\sum_{n=1}^{+ \infty} n a_{n} \left[ F\left(\frac{1}{2},\frac{1}{2},1;b^2\right)-\frac{b^2}{2} F\left(\frac{1}{2},\frac{3}{2},2;b^2\right) \right]\omega^{n+1}.\]
Moreover
\begin{align*}
\sum_{n=1}^{+ \infty} \left(n+\frac{1}{2}\right) \overline{\omega} \fint_{\T}  \frac{ a_n \overline{\omega}^n- c_n \overline{\tau}^n}{\vert \omega-b\tau\vert }\frac{d\tau}{\tau}=\sum_{n=1}^{+ \infty} \left(n+\frac{1}{2}\right) \left[a_n F\left(\frac{1}{2},\frac{1}{2},1;b^2\right)-c_n b^n \frac{(\frac{1}{2})_n}{n!}F\left(\frac{1}{2},n+\frac{1}{2},n+1;b^2\right) \right] \overline{\omega}^{n+1}.
\end{align*}
Now, using $(\ref{c7})$ we obtain
\[\frac{b}{2}\sum_{n=1}^{+ \infty} \overline{\omega} \fint_{\T} \frac{(b \omega- \tau)(a_n \omega^n - c_n \tau^n)}{\vert b\tau-\omega \vert^3} d\tau =\frac{b^2}{2} \Bigg[ c_n b^n \frac{\left( \frac{3}{2}\right)_{n+1}}{(n+1)!} F \left( \frac{1}{2}, n+ \frac{5}{2},n+2; b^2 \right)- \frac{3}{2}a_n F\left( \frac{1}{2},\frac{5}{2},2; b^2 \right)\Bigg] \omega^{n+1} .\]
For the last term of $\mathcal{L}_2(h_1,h_2)$ we use $(\ref{c5})$ and $(\ref{c6})$
\begin{align*}
& -\frac{1}{2}\sum_{n=1}^{+ \infty} \fint_{\T} \frac{(b - \omega \overline{\tau})(a_n \omega^n - c_n \tau^n)}{\vert b\tau-\omega \vert^3} d \tau= -\frac{1}{2}\sum_{n=1}^{+ \infty}  a_n  \Bigg[ \frac{3}{2} b^2 F\left( \frac{3}{2},\frac{5}{2},2; b^2 \right)-F\left( \frac{3}{2},\frac{3}{2},1; b^2 \right)  \Bigg] \omega^{n+1}\\
 & +\frac{1}{2}\sum_{n=1}^{+ \infty} c_n \frac{b^n}{n!} \left( \frac{3}{2}\right)_{n} \Bigg[ \frac{\left( n+ \frac{3}{2}\right)}{(n+1)} b^2 F\left( \frac{3}{2},n+\frac{5}{2},n+2; b^2 \right)-F\left( \frac{3}{2},n+\frac{3}{2},n+1; b^2 \right)\Bigg] \omega^{n+1} .
 \end{align*}
Now we shall apply  $(\ref{eq2})$ with $a= \frac{1}{2}$, $b=\tilde{n}+ \frac{3}{2}$, $c=\tilde{n}+1$ and $z= b^2$ where $\tilde{n} \in \lbrace 0,n \rbrace$
\[ -\frac{1}{2}\sum_{n=1}^{+ \infty} \fint_{\T} \frac{(b - \omega \overline{\tau})(a_n \omega^n - c_n \tau^n)}{\vert b\tau-\omega \vert^3} d \tau =\frac{1}{2}\sum_{n=1}^{+ \infty} \left[ a_n   F \left( \frac{1}{2},  \frac{3}{2},1; b^2 \right)-c_n\frac{b^n}{n!} \left( \frac{3}{2}\right)_{n} F \left( \frac{1}{2},n+  \frac{3}{2},1; b^2 \right)\right] \omega^{n+1}.\]
Finally we get
\[\mathcal{L}_2(h_1,h_2)(\omega)=\frac{i}{2}\sum_{n=1}^{+ \infty} \left[ a_{n}(\tilde{\gamma_{n}}-\gamma_{n})+(\tilde{\beta_{n}}-\beta_{n})c_{n}\right](\omega^{n+1}-\overline{\omega}^{n+1} )\]
with
\begin{align*}
\gamma_{n}=&n \left[F\left(\frac{1}{2},\frac{1}{2},1;b^2 \right)-\frac{b^2}{2}F \left(\frac{1}{2},\frac{3}{2},2;b^2 \right) \right]-\frac{3}{4}b^2F \left(\frac{1}{2},\frac{5}{2},2;b^2 \right)+\frac{1}{2} F\left(\frac{1}{2},\frac{3}{2},1;b^2\right),\\
\tilde{\gamma_{n}}=& (n+\frac{1}{2})F\left(\frac{1}{2},\frac{1}{2},1;b^2\right),\\
\beta_{n}=&\frac{b^{n+2}}{2} \frac{(\frac{3}{2})_{n+1}}{(n+1)!}F\left(\frac{1}{2},n+\frac{5}{2},n+2;b^2\right)-\frac{b^{n}}{2} \frac{(\frac{3}{2})_{n}}{n!}F\left(\frac{1}{2},n+\frac{3}{2},n+1;b^2\right)  \\
\intertext{ and }
\tilde{\beta_{n}}=&-(n+\frac{1}{2})b^{n} \frac{(\frac{1}{2})_{n}}{n!}F\left(\frac{1}{2},n+\frac{1}{2},n+1;b^2\right).
\end{align*}
Now,  we want to  simplify the expression of $\mathcal{L}_2 (h_1,h_2)$ through the use of the identities $(\ref{eq2})$-$ (\ref{eq4} )$. We begin with $(\ref{eq11} )$ with $a =\frac{1}{2}$, $b=\tilde{n}+\frac{1}{2}$, $c= \tilde{n}+1$ and $z=b^2$ where $\tilde{n} \in \lbrace 0,n \rbrace$ which implies
\[\gamma_{n}-\tilde{\gamma_n}=-n\frac{b^2}{2}F\left(\frac{1}{2},\frac{3}{2},2;b^2\right)-\frac{3}{4}b^2F\left(\frac{1}{2},\frac{5}{2},2;b^2\right)+\frac{b^2}{4}F\left(\frac{3}{2},\frac{3}{2},2;b^2\right)\]
and 
\[\beta_{n}-\tilde{\beta_{n}}=\frac{b^{n+2}}{2 (n+1)!} \left( \frac{3}{2} \right)_{n} \Bigg[ \left( n+ \frac{3}{2} \right) F\left(\frac{1}{2},n+\frac{5}{2},n+2;b^2\right)-\frac{1}{2}F\left(\frac{1}{2},n+ \frac{3}{2},n+1;b^2\right) \bigg]. \]
Thus using $(\ref{eq12})$, one may check the following expression
\begin{align*}
\mathcal{L}_2(h_1,h_2)(\omega)&=\frac{i}{2}\sum_{n=1}^{+ \infty} (n+1)\frac{b^2}{2}F\left(\frac{1}{2},\frac{3}{2},2;b^2\right)a_{n}(\omega^{n+1}-\overline{\omega}^{n+1} )\\
&-\frac{i}{2}\sum_{n=1}^{+ \infty}\frac{b^{n+2}}{n!}(\frac{1}{2})_{n+1}F\left(\frac{1}{2},n+\frac{3}{2},n+2;b^2\right)c_{n}(\omega^{n+1}-\overline{\omega}^{n+1} ).
\end{align*}
Now we focus on $\mathcal{L}_3(h_1,h_2)$ given by
\begin{align*}
\mathcal{L}_3(h_1,h_2)(\omega)&=\textnormal{Im}\left\lbrace -\overline{\omega}h_2'(\overline{\omega})\fint_{\T}  \frac{\tau-b\omega }{\vert \tau-b\omega\vert }\frac{d\tau}{\tau}+\overline{\omega}b\fint_{\T}  \frac{\omega h_2'(\omega)-\tau h_1'(\tau)}{\vert b\omega-\tau\vert }\frac{d\tau}{\tau}\right\rbrace \\
&+\textnormal{Im} \left\lbrace b\overline{\omega}\fint_{\T}  \frac{(\tau-b\omega )\textnormal{Re}\Big(\big(h_1(\tau)-h_2(\omega)\big) \big(\overline{\tau-b\omega}\big)\Big)}{\vert \tau-b\omega \vert^3} \frac{d\tau}{\tau} \right\rbrace .
\end{align*}
Observe that
\begin{align*} 
b\overline{\omega}\fint_{\T}  \frac{(\tau-b\omega )\textnormal{Re}\Big(\big(h_1(\tau)-h_2(\omega)\big)\big(\overline{\tau-b\omega}\big)\Big)}{\vert \tau-b\omega \vert^3} \frac{d\tau}{\tau}&=\frac{b}{2}\overline{\omega}\fint_{\T}  \frac{h_1(\tau)-h_2(\omega)}{\vert \tau-b\omega \vert} \frac{d\tau}{\tau}\\
&+\frac{b}{2}\overline{\omega}\fint_{\T} \frac{\Big(\tau-b\omega \Big)^2 \Big(\overline{h_1(\tau)-h_2(\omega)}\Big)}{\vert \tau-b\omega \vert^3} \frac{d\tau}{\tau}.
\end{align*}
Replacing $h_j$  and $h_j'$ by their expressions we get
\begin{align*}
\mathcal{L}_3(h_1,h_2)(\omega)&=\textnormal{Im}\left\lbrace  \sum_{n=1}^{+ \infty}n c_n \omega^n \fint_{\T}  \frac{\tau-b\omega }{\vert \tau-b\omega\vert }\frac{d\tau}{\tau}-\overline{\omega}b \sum_{n=1}^{+ \infty} \left( n+ \frac{1}{2} \right) \fint_{\T}  \frac{ c_n \overline{\omega}^n- a_n \overline{\tau}^n}{\vert b\omega-\tau\vert }\frac{d\tau}{\tau}\right\rbrace \\
&+\textnormal{Im}\left\lbrace \frac{b}{2}\overline{\omega} \sum_{n=1}^{+ \infty}\fint_{\T} \frac{\Big(b\omega- \tau \Big) \Big(  c_n \omega^n- a_n \tau^n\Big)}{\vert \tau-b\omega \vert^3} d\tau +\frac{b^2}{2} \sum_{n=1}^{+ \infty}\fint_{\T} \frac{\Big(1-b\omega \overline{\tau} \Big)\Big(  c_n \omega^n- a_n \tau^n\Big)}{\vert \tau-b\omega \vert^3}d\tau \right\rbrace.
\end{align*}
As before, we shall  split the computations in many parts. Thanks to  $(\ref{c3})$ and $(\ref{c4})$, the first term takes the form
\begin{align*}
\sum_{n=1}^{+ \infty}n c_n \omega^n \fint_{\T} \frac{\tau-b\omega }{\vert \tau-b\omega\vert }\frac{d\tau}{\tau} & =\sum_{n=1}^{+ \infty}n c_n \omega^n\fint_{\T} \frac{\tau-b\omega }{\vert \tau b-\omega\vert }\frac{d\tau}{\tau}\\
&=\sum_{n=1}^{+ \infty}n bc_n  \Big[ \frac{1}{2}F\left(\frac{1}{2},\frac{3}{2},2;b^2\right)-F\left(\frac{1}{2},\frac{1}{2},1;b^2\right)\Big]\omega^{n+1}.
\end{align*}
Note that we have used in the first line the identity
$$
\forall \tau, \omega\in \T,\quad |\tau-b\omega|=|b\tau-\omega|
$$
For the second term, we use $(\ref{c3})$ and $(\ref{c4})$ to obtain 
\begin{align*}
-\overline{\omega}b \sum_{n=1}^{+ \infty} \left( n+ \frac{1}{2} \right) \fint_{\T}  \frac{ c_n \overline{\omega}^n- a_n \overline{\tau}^n}{\vert b\omega-\tau\vert }\frac{d\tau}{\tau}&=- \sum_{n=1}^{+ \infty} \left( n+ \frac{1}{2} \right) b c_n  F\left(\frac{1}{2},\frac{1}{2},1;b^2\right) \overline{\omega}^{n+1}\\
&+\sum_{n=1}^{+ \infty} \left( n+ \frac{1}{2} \right) a_n \frac{b^{n+1}}{n!} \left( \frac{1}{2} \right)_n   F\left(\frac{1}{2},n+\frac{1}{2},n+1;b^2\right)\overline{\omega}^{n+1}.
\end{align*}
The computation of the third term can be done in view of $(\ref{c8})$,  \begin{align*}
\frac{b}{2}\overline{\omega} \sum_{n=1}^{+ \infty}\fint_{\T} \frac{\Big(b\omega-\tau \Big) \Big(  c_n \omega^n- a_n \tau^n\Big)}{\vert \tau-b\omega \vert^3} d\tau & =\sum_{n=1}^{+ \infty}  \frac{a_n}{2} \frac{b^{n+3}}{(n+2)!} \left( \frac{1}{2} \right)_{n+2} F\left(\frac{3}{2},n+\frac{5}{2},n+3;b^2\right)\omega^{n+1}\\
&-\sum_{n=1}^{+ \infty} \frac{3b^3}{16}c_n F\left(\frac{3}{2},\frac{5}{2},3;b^2\right) \omega^{n+1} .
\end{align*}
For the last term, we use $(\ref{c5})$ and $(\ref{c6})$
\begin{align*}
&\frac{b^2}{2} \sum_{n=1}^{+ \infty}\fint_{\T} \frac{\Big(1-b\omega \overline{\tau} \Big)\Big(  c_n \omega^n- a_n \tau^n\Big)}{\vert \tau-b\omega \vert^3}d\tau=\sum_{n=1}^{+ \infty} \frac{b^3}{2}c_n  \Bigg[ \frac{3}{2} F\left(\frac{3}{2},\frac{5}{2},2;b^2\right)-F\left(\frac{3}{2},\frac{3}{2},1;b^2\right)\Bigg]\omega^{n+1} \\
&-\sum_{n=1}^{+ \infty} \frac{b^{n+3}}{2}a_n  \Bigg[ \frac{\left( \frac{3}{2}\right)_{n+1}}{(n+1)!}F\left(\frac{3}{2},n+\frac{5}{2},n+2;b^2\right) -\frac{\left( \frac{3}{2}\right)_{n}}{n!}F\left(\frac{3}{2},n+\frac{3}{2},n+1;b^2\right) \Bigg]\omega^{n+1}
\end{align*}
Finally, we obtain the following expression
\[\mathcal{L}_3(h_1,h_2)(\omega)=\frac{i}{2} \sum_{n=1}^{+ \infty}  \left[a_{n}(\tilde{\Delta}_{n}-\Delta_{n})+c_{n}(\tilde{\theta}_{n}-\theta_{n})\right](\omega^{n+1}-\overline{\omega}^{n+1} )\]
with 
\begin{align*}
\Delta_{n}&=\frac{b^{n+3}}{2}  \frac{(\frac{1}{2})_{n+2}}{(n+2)!}F\left(\frac{3}{2},n+\frac{5}{2},n+3;b^2\right)-\frac{b^{n+3}}{2}   \frac{(\frac{3}{2})_{n+1}}{(n+1)!}F\left(\frac{3}{2},n+\frac{5}{2},n+2;b^2\right)\\
& +\frac{b^{n+3}}{2}\frac{(\frac{3}{2})_{n}}{n!}F\left(\frac{3}{2},n+\frac{3}{2},n+1;b^2\right),\\
\tilde{\Delta}_{n}&=\left(n+\frac{1}{2} \right)b^{n+1} \frac{(\frac{1}{2})_{n}}{n!}F\left(\frac{1}{2},n+\frac{1}{2},n+1;b^2\right),\\
\theta_{n}&= bn \left[  \frac{1}{2}F\left(\frac{1}{2},\frac{3}{2},2;b^2\right)- F\left(\frac{1}{2},\frac{1}{2},1;b^2\right)\right] -\frac{3b^3}{16} F\left(\frac{3}{2},\frac{5}{2},3;b^2\right)\\
&+ \frac{b^3}{2}\Bigg[ \frac{3}{2} F\left(\frac{3}{2},\frac{5}{2},2;b^2\right)-F\left(\frac{3}{2},\frac{3}{2},1;b^2\right)\Bigg]\\
\intertext{ and }
\tilde{\theta}_{n}&= -\left(n+\frac{1}{2} \right)bF\left(\frac{1}{2},\frac{1}{2},1;b^2\right).
\end{align*}
We want to simplify the expression of $\mathcal{L}_3(h_1,h_2)$. First we note that
\begin{align*}
\theta_n-\tilde{\theta}_n&= \frac{nb}{2} F\left(\frac{1}{2},\frac{3}{2},2;b^2\right)+\frac{b}{2}F\left(\frac{1}{2},\frac{1}{2},1;b^2\right)-\frac{3b^3}{16}F\left(\frac{3}{2},\frac{5}{2},3;b^2\right)\\
&+\frac{b^3}{2} \Bigg[ \frac{3}{2} F\left(\frac{3}{2},\frac{5}{2},2;b^2\right)-F\left(\frac{3}{2},\frac{3}{2},1;b^2\right)\Bigg].
\end{align*}
By using $(\ref{eq11})$ we get 
\begin{align*}
\theta_n-\tilde{\theta}_n&= \frac{b}{2}(n+1) F\left(\frac{1}{2},\frac{3}{2},2;b^2\right)+\frac{b}{2}F\left(\frac{1}{2},\frac{1}{2},1;b^2\right)-\frac{b}{2} F\left(\frac{1}{2},\frac{3}{2},1;b^2\right)\\
&+\frac{b^3}{2} \Bigg[ \frac{3}{2} F\left(\frac{3}{2},\frac{5}{2},2;b^2\right)-F\left(\frac{3}{2},\frac{3}{2},1;b^2\right)\Bigg].
\end{align*}
Now combining $(\ref{eq11})$ with $a=b=\frac{1}{2}$, $c=1$ and $z=b^2$ as well as  $(\ref{eq4})$ with  $a=b=\frac{3}{2}$, $c=1$ and $z=b^2$, one may obtain the following identity
\[  F\left(\frac{1}{2},\frac{1}{2},1;b^2\right)- F\left(\frac{1}{2},\frac{3}{2},1;b^2\right)+\frac{3}{2}b^2  F\left(\frac{3}{2},\frac{5}{2},2;b^2\right)-b^2 F\left(\frac{3}{2},\frac{3}{2},1;b^2\right)=0.\]
Consequently, 
\[\theta_n-\tilde{\theta}_n=\frac{b}{2}(n+1)F\left(\frac{1}{2},\frac{3}{2},2;b^2\right).\]
On the other hand
\begin{align*}
\Delta_n-\tilde{\Delta}_n&= b^{n+1} \frac{(\frac{1}{2})_{n+1}}{(n+1)!} \Bigg[\frac{b^2}{2} \frac{(n+\frac{3}{2})}{(n+2)}F\left(\frac{3}{2},n+\frac{5}{2},n+3;b^2\right) -b^2  \left( n+ \frac{3}{2} \right) F\left(\frac{3}{2},n+\frac{5}{2},n+2;b^2\right)\\
&+b^2 (n+1)F\left(\frac{3}{2},n+\frac{3}{2},n+1;b^2\right)- (n+1) F\left(\frac{1}{2},n+\frac{1}{2},n+1;b^2\right)\Bigg].
\end{align*}
Applying  $(\ref{eq11})$ with $a=\frac{1}{2},b=n+\frac{1}{2},c=n+1$ and $z=b^2$ one gets,
\begin{align*}
\Delta_n-\tilde{\Delta}_n&= b^{n+1} \frac{(\frac{1}{2})_{n+1}}{(n+1)!} \Bigg[\frac{b^2}{2} \frac{(n+\frac{3}{2})}{(n+2)}F\left(\frac{3}{2},n+\frac{5}{2},n+3;b^2\right) -b^2  \left( n+ \frac{3}{2} \right) F\left(\frac{3}{2},n+\frac{5}{2},n+2;b^2\right)\\
&+b^2 (n+1)F\left(\frac{3}{2},n+\frac{3}{2},n+1;b^2\right)- (n+1) F\left(\frac{1}{2},n+\frac{3}{2},n+1;b^2\right)+\frac{b^2}{2}F\left(\frac{3}{2},n+\frac{3}{2},n+2;b^2\right)\Bigg].
\end{align*}
Again applying $(\ref{eq11})$ with $a=\frac{1}{2},b=n+\frac{3}{2},c=n+2$ and $z=b^2$,  we deduce
\begin{align*}
\Delta_n-\tilde{\Delta}_n&= b^{n+1} \frac{(\frac{1}{2})_{n+1}}{(n+1)!} \Bigg[(n+\frac{3}{2})F\left(\frac{1}{2},n+\frac{5}{2},n+2;b^2\right) -(n+\frac{3}{2})F\left(\frac{1}{2},n+\frac{3}{2},n+2;b^2\right) \\
&-b^2  \left( n+ \frac{3}{2} \right) F\left(\frac{3}{2},n+\frac{5}{2},n+2;b^2\right)+b^2 (n+1)F\left(\frac{3}{2},n+\frac{3}{2},n+1;b^2\right)\\
&- (n+1) F\left(\frac{1}{2},n+\frac{3}{2},n+1;b^2\right)+\frac{b^2}{2}F\left(\frac{3}{2},n+\frac{3}{2},n+2;b^2\right)\Bigg].
\end{align*}
We use $(\ref{eq4})$ with $a=\frac{3}{2},b=n+\frac{3}{2},c=n+1$ and $z=b^2$ to cancel some terms 
\begin{align*}
\Delta_n-\tilde{\Delta}_n&= b^{n+1} \frac{(\frac{1}{2})_{n+1}}{(n+1)!} \Bigg[(n+\frac{3}{2})F\left(\frac{1}{2},n+\frac{5}{2},n+2;b^2\right) -(n+\frac{3}{2})F\left(\frac{1}{2},n+\frac{3}{2},n+2;b^2\right) \\
&- (n+1) F\left(\frac{1}{2},n+\frac{3}{2},n+1;b^2\right)\Bigg].
\end{align*}
Finally, using $(\ref{eq4})$ with $a=\frac{1}{2},b=n+\frac{3}{2},c=n+1$ and $z=b^2$, we obtain
\[\Delta_n-\tilde{\Delta}_n= -\frac{b^{n+1}}{(n+1)!}\left(\frac{1}{2}\right)_{n+1} (n+1)F\left(\frac{1}{2},n+\frac{3}{2},n+2;b^2\right).\]
Consequently, we have
\begin{align*}
\mathcal{L}_3(h_1,h_2)(\omega)=&\frac{i}{2}\sum_{n=1}^{+ \infty} \frac{b^{n+1}}{(n+1)!}\left(\frac{1}{2}\right)_{n+1} (n+1)F\left(\frac{1}{2},n+\frac{3}{2},n+2;b^2\right)a_{n}(\omega^{n+1}-\overline{\omega}^{n+1} )\\
&-\frac{i}{2}\sum_{n=1}^{+ \infty}\frac{b}{2}(n+1)F\left(\frac{1}{2},\frac{3}{2},2;b^2\right)c_{n}(\omega^{n+1}-\overline{\omega}^{n+1} ).
\end{align*}
As we have (see \cite{7})
\[\Lambda_n(b)=\frac{\left(\frac{1}{2}\right)_n}{n!}b^{n-1}F\left(\frac{1}{2},n+\frac{1}{2},n+1,b^2\right).\]
the proof of the proposition is now achieved.
\end{proof}
\subsection{Monotonicity of the eigenvalues}
 In what follows we shall use the variable $\lambda \triangleq 1-2\Omega$ instead of $\Omega$. The main task is to list the suitable conditions on the used parameters in order to guarantee a one-dimensional kernel. Recall from Proposition \ref{Matrice} that the operator $DG(\Omega,0,0)$ acts as a  Fourier matrix multiplier 
 and the determinant of each matrix $M_n$ is given by
 \begin{equation}\label{dis}
\textnormal{det}(M_n)=\frac{b}{4}(\lambda^2-2C_n\lambda+D_n)
\end{equation}
with\[ \left\lbrace\begin{array}{lcl}
C_n&=&1+\left(\frac{1}{b}-1\right)S_n-(1-b^2)\Lambda_1(b)\\
D_n&=&-\frac{4}{b}S_n^2+2 \left[\frac{1}{b}-1+2(1+b)\Lambda_1(b)\right]S_n-4b^2\left(\Lambda_1^2(b)-\Lambda_n^2(b)\right)-2(1-b^2)\Lambda_1(b)+1\\
\end{array}\right. .
\]
From that proposition one can easily see that the kernel of $DG(\Omega,0,0)$ is non trivial if and only if
$$
\big\{\exists n\geq2,\, \textnormal{det}(M_n)=0\big\}
$$
Therefore the dimension of the kernel is related to   the structure of the eigenvalues and to how they depend on  the frequency modes. 
Observe that $\lambda\mapsto  \textnormal{det}(M_n)$ is  a second order polynomial  and the roots structure  depends on the reduced discriminant which is given by
\[\Delta_n=\left(\frac{1}{b}+1\right) S_n-(1+b^2)\Lambda_1(b) ^2-4b^2\Lambda_n^2(b).\]
We shall prove the following proposition.
\begin{proposition}\label{vp}
\begin{itemize}
\item[(1)] For any $ n \in N^*$ we have $\Lambda_n(b) \geq 0$, $n \mapsto S_n$ is a  strictly increasing sequence, $n\mapsto \Lambda_n(b)$ is a strictly decreasing sequence and $b\mapsto \Lambda_n(b)$ is a strictly increasing function.
\item[(2)] There exists $N \geq 2$ such that for any $n \geq N$ we get $\Delta_n >0$ and the equation $\textnormal{det}(M_n)=0$ admits two different real solutions given by
\[ \lambda_n^{\pm}=C_n\pm \sqrt{\Delta_n}.\]
\item[(3)]The sequences $(\Delta_n)_{n\geq N}$ and $(\lambda_n^+)_{n\geq N}$ are strictly increasing and  $(\lambda_n^-)_{n\geq N}$ is strictly decreasing.
\item[(4)] $\forall m >n>N$ we have
\[ \lambda_m^-<\lambda_n^-<\lambda_n^+<\lambda_m^+. \]
\end{itemize}
\end{proposition}
\begin{proof}
$(1)$ The positivity and the monotonicity of  $\Lambda_n(b)$ follow easily from the integral representation
 \[\Lambda_n(b)=\frac{b^{n-1}}{\Gamma^2(\frac{1}{2})}\int_0^1 x^{n-\frac{1}{2}}(1-x)^{-\frac{1}{2}}(1-b^2x)^{-\frac{1}{2}} dx \text{, for }b\in(0,1).\]
 As to the monotonicity of $S_n$ it is obvious. \\
$(2)$  We write $\Delta_n(b)=E_n(b)F_n(b)$
with
\[ \left\lbrace\begin{array}{lcl}
E_n(b)&=&\left(\frac{1}{b}+1 \right)S_n-(1+b^2)\Lambda_1(b)-2b\Lambda_n(b)\\
F_n(b)&=&\left(\frac{1}{b}+1\right)S_n-(1+b^2)\Lambda_1(b)+2b\Lambda_n(b)\\
\end{array}\right..
\]
We remark that
\[ \Delta_n(b)>0  \textnormal{ if and only if } E_n(b)>0 \text{ or } F_n(b)<0. \]
Using the strictly monotonicity of the sequences $(\Lambda_n)_{n\in \N*}$ and  $(S_n)_{n\in \N*}$  we get
\[E_{n+1}(b)-E_n(b)=\left(1+\frac{1}{b}\right)(S_{n+1}-S_n)-2b \big( \Lambda_{n+1}(b)-\Lambda_n(b)\big)>0.\]
Therefore $(E_n(b))_{n\in \N*}$ is a strictly increasing.
As $\underset{n\mapsto + \infty}{\lim}E_n(b)=+\infty$ and 
\[E_1(b)=-(1+b)^2\Lambda_1(b),\] we obtain that 
\[ \exists N \in \N^{\star} \text{such that } \forall n \geq N , \; E_n(b)>0. \]
This implies the assertion $(2)$.

$(3)$ Straightforward computations yield
\begin{align*}
\Delta_{n+1}-\Delta_n&=\left(1+\frac{1}{b}\right)^2(S_{n+1}^2-S_n^2)-4b^2\left(\Lambda_{n+1}^2(b)-\Lambda_n^2(b)\right)-2(1+b^2)\left(1+\frac{1}{b}\right)\Lambda_1(b)(S_{n+1}-S_n)\\
&=\left(1+\frac{1}{b}\right)(S_{n+1}-S_n)\bigg[\left(1+\frac{1}{b}\right)(S_{n+1}+S_n)-2(1+b^2)\Lambda_1(b)\bigg]-4b^2\left(\Lambda^2_{n+1}(b)-\Lambda^2_n(b)\right)\\
&>\left(1+\frac{1}{b}\right)(S_{n+1}-S_n)\bigg[\left(1+\frac{1}{b}\right)(S_{n+1}+S_n)-2(1+b^2)\Lambda_1(b) \bigg]\\
&>\left(1+\frac{1}{b}\right)(S_{n+1}-S_n) \bigg[\left(1+\frac{1}{b}\right)(S_{n+1}+S_n)-2(1+b^2)\Lambda_1(b)-2b \Big[\Lambda_{n+1}(b)+\Lambda_n(b)\Big] \bigg]\\
&>\left(1+\frac{1}{b}\right)(S_{n+1}-S_n)\big(E_{n+1}(b)+E_n(b) \big)>0 .
\end{align*}

As $(\Delta_n)_{n\geq n}$ and $(S_n)_{n\in \N^{\star}}$ are strictly increasing,  $(\lambda_n^+)_{n\geq N}$ is also strictly increasing.
We focus now on  $\lambda_n^-$:
\begin{align*}
\lambda_{n+1}^--\lambda_n^-&=\left(\frac{1}{b}-1\right)(S_{n+1}-S_n)-\big[\sqrt{\Delta_{n+1}}-\sqrt{\Delta_n}\big]\\
&=\left(\frac{1}{b}-1\right)(S_{n+1}-S_n)-\frac{\Delta_{n+1}-\Delta_n}{\sqrt{\Delta_{n+1}}+\sqrt{\Delta_n}}\\
&=-\frac{1}{\sqrt{\Delta_{n+1}}+\sqrt{\Delta_n}} \left[\left(1+\frac{1}{b}\right)(S_{n+1}-S_n)\left[\left(1+\frac{1}{b}\right)(S_{n+1}+S_n)-2(1+b^2)\Lambda_1(b)\right] \right]\\
&+\frac{4b^2(\Lambda_{n+1}^2(b)-\Lambda_n^2(b))}{\sqrt{\Delta_{n+1}}+\sqrt{\Delta_n}}+\left(\frac{1}{b}-1\right)(S_{n+1}-S_n)\\
&<-\frac{1}{\sqrt{\Delta_{n+1}}+\sqrt{\Delta_n}} \left[\left(1+\frac{1}{b}\right)(S_{n+1}-S_n)\left[\left(1+\frac{1}{b}\right)(S_{n+1}+S_n)-2(1+b^2)\Lambda_1(b)\right] \right]\\
&+\left(\frac{1}{b}-1\right)(S_{n+1}-S_n)\\
&<\frac{1}{b}\left[1-\frac{1}{\sqrt{\Delta_{n+1}}+\sqrt{\Delta_n}} \left[\left(1+\frac{1}{b}\right)(S_{n+1}+S_n)-2(1+b^2)\Lambda_1(b) \right]\right](S_{n+1}-S_n)\\
&-\left[1+\frac{1}{\sqrt{\Delta_{n+1}}+\sqrt{\Delta_n}} \left[\left(1+\frac{1}{b}\right)(S_{n+1}+S_n)-2(1+b^2)\Lambda_1(b) \right]\right](S_{n+1}-S_n)\\
&<0
\end{align*}
because $\big[\left(1+\frac{1}{b}\right)(S_{n+1}+S_n)-2(1+b^2)\Lambda_1(b) \big]>E_{n+1}(b)+E_n(b)>0$ and $\sqrt{\Delta_n}<\left(1+\frac{1}{b}\right)S_n-(1+b^2)\Lambda_1(b)$.

Consequently the sequence $(\lambda_n^-)_{n\geq N}$ is strictly decreasing.

$(4)$  It is obvious and follows from $(2)$ and $(3)$.
\end{proof}
\subsection{Proof of Theorem \ref{existence}}\label{Bifurcation at simple eigenvalues}\label{conclusion}
This section is dedicated to the proof of the main result of this paper which is deeply related to the spectral study developed in the preceding section combined with Crandall-Rabinowitz's theorem.  To proceed, fix $b \in (0,1)$ and $m\geq N$, where $N$ was defined in Proposition $\ref{vp}$. Set,
\begin{eqnarray*}
X^{k+\log}_m=  X^{k+\log}\cap \mathcal{A}_\varepsilon^m.
\end{eqnarray*}
 We define the ball of radius $r \in (0,1)$ by 
 \[ B_r^m = \Big\{ f \in X^{k+\log}_m, \Vert f \Vert_{X^{k+\log}_m} \leq r \Big\}\]
 and we introduce the neighborhood of the trivial solution $(0,0)$, 
 \[ V_{m,r}\triangleq B_r^m \times B_r^m. \]
 The set $V_{m,r}$ is endowed with the induced topology of the product spaces.
 Take $(f_1,f_2) \in V_{m,r}$ then the expansions of the associated conformal mappings $\Phi_1,\Phi_2$ in $\mathbb{\Delta}_{\varepsilon} $ are given successively by
 \[ \Phi_1(z)=z+ f_1(z)=z\left( 1+\sum_{n=1}^{+ \infty} \frac{a_n}{z^{nm}}\right)\]
 and
 \[ \Phi_2(z)=bz+ f_2(z)=z\left( b+\sum_{n=1}^{+ \infty} \frac{c_n}{z^{nm}}\right).\]
Consequently for any $z \in \mathbb{\Delta}_{\varepsilon} $
\begin{equation}\label{mprem}
 \Phi_j(e^{\frac{2i\pi}{m}}z) = e^{\frac{2i\pi}{m}} \Phi_j(z) \text{, }j=1,2 \text{ and } \vert z \vert > \varepsilon .
 \end{equation}
 
From Proposition $\ref{vp}$ recall the definition of the eigenvalues $\lambda_m^{\pm}$ and the associated angular velocities are
\begin{align*}
\Omega_m^{\pm}&= \frac{1}{2}-\frac{1}{2}\lambda_m^{\pm}\\
&=\frac{1}{2}\tilde{C}_m \pm \frac{1}{2} \sqrt{\Delta_m}
\end{align*}
with
\[\Delta_m=\bigg(\Big(\frac{1}{b}+1\Big)S_m-(1+b^2)\Lambda_1(b)\bigg)^2-4b^2\Lambda_m^2(b)\]
and
\[\tilde{C}_m=\left(1-\frac{1}{b}\right)S_m+(1-b^2)\Lambda_1(b).\]
Note that $S_m$ and $\Lambda_m(b)$ were introduced in Proposition $\ref{Matrice}$. The V-states equations are described in $(\ref{G_j})$ and $(\ref{V-state})$ which we restate here, for $ j \in \lbrace 1,2 \rbrace$,
\[\tilde{G}(\Omega,\Phi_1,\Phi_2) \triangleq G(\Omega,f_1,f_2) \text{ and  } G=(G_1,G_2) \]
with
\[\tilde{G}_j(\Omega,\Phi_1,\Phi_2)(\omega)= \textnormal{Im} \left\lbrace \left(\Omega\Phi_j(\omega)- \fint_{\T} \frac{\tau\Phi_1'(\tau)-\omega \Phi_j'(\omega)}{\vert \Phi_1(\tau)-\Phi_j(\omega)\vert }\frac{d\tau}{\tau}+ \fint_{\T} \frac{\tau\Phi_2'(\tau)-\omega \Phi_j'(\omega)}{\vert \Phi_2(\tau)-\Phi_j(\omega)\vert }\frac{d\tau}{\tau}\right)\overline{\Phi_j'(\omega)}\overline{\omega} \right\rbrace. \]
The following result is more precise than Theorem \ref{existence}.
\begin{theorem}\label{appli}
Let $k\geq3, N$ be as in the Proposition $\ref{vp}$, $m\geq N$, and take $\Omega\in \lbrace \Omega_m^{\pm} \rbrace$. Then, the following assertions hold true.
\begin{itemize}
\item[(1)] There  exists $r>0$ such that $G: \R \times V_{m,r} \mapsto Y^{k-1}_m \times Y^{k-1}_m$ is well- defined and is of class $C^1$.
\item[(2)] The kernel of $DG(\Omega,0,0)$ is one dimensional and generated by
\[  v_{0,m}:\omega\in T \mapsto \left(\begin{array}{c}
\Omega+\frac{S_m}{b}-\Lambda_1(b) \\ 
-\Lambda_m(b)
\end{array} \right) \overline{\omega}^{m-1}.\]
\item[(3)] The range of $DG(\Omega,0,0)$ is closed and is of co-dimension one in $Y^{k-1}_m \times Y^{k-1}_m$.
\item[(4)] Transversality assumption: If $\Omega$ is a simple eigenvalue $(\Delta_m>0)$ then
\[ \partial_{\Omega}DG(\Omega_m^{\pm},0,0)v_{0,m} \notin \textnormal{Im}\left(DG(\Omega_m^{\pm},0,0)\right).\]
\end{itemize}
\end{theorem}
\begin{proof}
$(1)$ Compared to Theorems $\ref{biendef}$ and $\ref{regularite}$, we just need to check that $G=(G_1,G_2)$ preserves the $m$-fold symmetry and maps $X_m^{k+\log} \times X_m^{k+\log}$ into $Y^{k-1}_m \times Y^{k-1}_m$. To this end, it is sufficient to check that for given $(f_1,f_2)\in X_m^{k+\log} \times X_m^{k+\log}$ the Fourier coefficients of $\tilde{G}_j(\Omega,\Phi_1,\Phi_2)$ vanish at frequencies which are not integer multiple of $m$.  This amounts to proving that,
\[\tilde{G}_j(\Omega,\Phi_1,\Phi_2)(e^{i\frac{2\pi}{m}}\omega)=\tilde{G}_j(\Omega,\Phi_1,\Phi_2)(\omega) \text{, } \forall \omega \in \T \text{, }j=1,2. \]
As
\begin{equation}\label{mder}
\Phi_j'(e^{\frac{2i\pi}{m}}\omega)=\Phi_j'(\omega),
 \end{equation}
the property is obvious for the first term $\textnormal{Im} \lbrace \Omega\overline{\omega}\overline{\Phi_j'(\omega)}\Phi_j(\omega) \rbrace$. For the two last terms of $\tilde{G}_j$ it is enough to check the identity,
 \[ \forall \omega \in \T, \; S(\Phi_i,\Phi_j)(e^{\frac{2i\pi}{m}} \omega)=e^{\frac{2i\pi}{m}}S(\Phi_i,\Phi_j)( \omega).\]
 This follows easily by making the change of variables $\tau =e^{\frac{2i\pi}{m}} \xi$ and from $(\ref{mprem})$ and $(\ref{mder})$, 
\begin{align*}
 S(\Phi_i,\Phi_j)(e^{\frac{2i\pi}{m}} \omega)&=\fint_{\T} \frac{e^{\frac{2i\pi}{m}} \xi \Phi_i'(e^{\frac{2i\pi}{m}} \xi)-e^{\frac{2i\pi}{m}} \omega \Phi_j'(e^{\frac{2i\pi}{m}}\omega)}{\vert \Phi_i(e^{\frac{2i\pi}{m}} \xi)-\Phi_j(e^{\frac{2i\pi}{m}}\omega)\vert }\frac{d\xi}{\xi}\\
 &=e^{\frac{2i\pi}{m}} \fint_{\T} \frac{ \xi \Phi_i'(\xi)- \omega \Phi_j'(\omega)}{\vert \Phi_i( \xi)-\Phi_j(\omega)\vert }\frac{d\xi}{\xi}\\
 &=e^{\frac{2i\pi}{m}} S(\Phi_i,\Phi_j)( \omega).
\end{align*}
This concludes the proof of the following statement,
\[ \forall\,(f_1,f_2)\in V_{m,r},\quad G(\Omega,f_1,f_2) \in Y_m^{k-1} \times Y^{k-1}_m. \]
$(2)$ We shall describe the kernel of linear operator $DG(\Omega_m^{\pm},0,0)$ and show that it is one-dimensional. Let $h_1,h_2$ be two functions in $X^{k+\log}_m$ such that 
\begin{equation}\label{fourier}
h_1(\omega)=\sum_{n=1}^{+ \infty}a_n\overline{\omega}^{nm-1} \text{ and }h_2(\omega)=\sum_{n=1}^{+ \infty}c_n\overline{\omega}^{nm-1}.
\end{equation}
Recall from Proposition $\ref{Matrice}$ the following expression,
\begin{equation}\label{diff2}
DG(\Omega,0,0)(h_1,h_2)=\frac{i}{2}\sum_{n\geq1}nmM_{nm} \left(\begin{array}{c}
a_n \\ 
c_n
\end{array}  \right)(\omega^{nm}-\overline{\omega}^{nm})
\end{equation}
where the matrice $M_n$ is given for $n\geq 2$ by :
\[ M_n= \left(\begin{array}{cc}
\Omega-S_n+b^2\Lambda_1(b) & -b^2\Lambda_n(b) \\ 
b\Lambda_n(b) & b\Omega+S_n-b\Lambda_1(b)
\end{array} \right).\]
Now if $\Omega\in \{ \Omega_m^{\pm}\}$ then
\[\textnormal{det}(M_m)=0.\]
Thus, the kernel of $DG(\Omega,0,0)$ is non trivial and is one-dimensional if and only if:
\[ \textnormal{det}(M_{nm}) \neq 0 \text{, } \forall n \geq 2.\]
This condition is ensured by Proposition $\ref{vp}$. Hence we have the equivalence:
\begin{equation}\label{equiv}
 (h_1,h_2) \in \textnormal{Ker}(DG(\Omega,0,0))\quad \textnormal{ if and only if }\quad  a_n=c_n=0 \; \forall n \geq 2 \text{ and } (a_1,c_1)\in \textnormal{Ker}(M_m)
 \end{equation}
Hence, a generator of $\textnormal{Ker}\left( DG(\Omega,0,0)\right)$ can be chosen as the pair of functions
\[ \omega\in \T \mapsto \left(\begin{array}{c}
\Omega+\frac{S_m}{b}-\Lambda_1(b) \\ 
-\Lambda_m(b)
\end{array} \right) \overline{\omega}^{m-1}.\]
$(3)$ We introduce 
\begin{align*}
Z_m=&\left\lbrace g=(g_1,g_2)\in Y^{k-1}_m \times Y^{k-1}_m | g(\omega)=\sum_{n \geq 1} \left(\begin{array}{c}
A_n \\ 
C_n
\end{array}\right) (\omega^{nm}-\overline{\omega}^{nm}), \, \forall \omega \in \T \right.\\
&\left. \text{s.t.} \,\,(A_n,C_n)\in \R^2 \text{ }\forall n\geq 2 \text{ and } \exists(a_1,c_1)\in \R^2 \text{ with  }M_m\left( \begin{array}{c}
a_1 \\ 
c_1
\end{array} \right)=\left( \begin{array}{c}
A_1 \\ 
C_1
\end{array} \right)  \right\rbrace.
\end{align*}
$Z_m$ is closed  and of codimension 1 in $Y^{k-1}_m \times Y^{k-1}_m$.
The following inclusion is obvious
\[\textnormal{Im}(DG(\Omega,0,0))\subset Z_m. \]
Therefore it remains just to check the converse. Let $(g_1,g_2)\in Z_m$, we shall  prove that the equation :
\[DG(\Omega,0,0)(h_1,h_2)=(g_1,g_2)\]
admits a solution $(h_1,h_2)\in X^{k+\log}_m\times X^{k+\log}_m$ where the Fourier expansions of these functions are given in $(\ref{fourier})$. According to $(\ref{diff})$, the preceding equation is equivalent to
\[ nm M_{nm}\left( \begin{array}{c}
a_n \\ 
c_n
\end{array} \right)=\left( \begin{array}{c}
A_n \\ 
C_n
\end{array} \right)  \text{, } \forall n \in \N^{\star}.\]
For $n=1$, the existence follows from the condition of space $Z_m$ and therefore we shall only focus \mbox{on
 $ n \geq 2 $.} Owing to $(\ref{equiv})$ the sequences $(a_n)_{n \geq 2}$ and $(c_n)_{n \geq 2}$ are uniquely determined by the formula
 \[ \left( \begin{array}{c}
a_n \\ 
c_n
\end{array} \right)=\frac{1}{nm} M_{nm}^{-1} \left( \begin{array}{c}
A_n \\ 
C_n
\end{array} \right)  \text{, } \forall n \geq2.\] 
 By computing the matrix $M_{nm}^{-1}$ we deduce that for all $n \geq 2$,
\[ \left\lbrace\begin{array}{lcl}
a_n&=&\frac{b(\Omega+\frac{1}{b}S_{nm}-\Lambda_1(b))}{mn\textnormal{det}(M_{nm})}A_n+\frac{b^2\Lambda_{nm}(b)}{mn\textnormal{det}(M_{nm})}C_n\\
c_n&=&-\frac{b\Lambda_{nm}(b)}{mn\textnormal{det}(M_{nm})}A_n+\frac{(\Omega-{b}S_{nm}+b^2\Lambda_1(b))}{mn \textnormal{det}(M_{nm})}C_n
\end{array}\right. .
\]
We just need to check that  $(h_1,h_2)\in  X^{k+\log}_m \times X^{k+\log}_m$. We shall develop the computations only for $h_1$ since the same analysis can be applied to $h_2$. By using  the characterization given by Lemma $\ref{caracterisation}$ one writes 
\begin{align*}
\|h_1\|_{X^{k+\log}}^2&\approx |a_1|^2+
\sum_{n=2}^{+ \infty} \frac{(mn)^{2k}}{\varepsilon^{2(nm+k-1)}} \left(1+\textnormal{log}(nm)\right)^2 \left[\frac{b\left(\Omega+\frac{1}{b}S_{nm}-\Lambda_1(b)\right)}{mn\textnormal{det}(M_{nm})}A_n+\frac{b^2\Lambda_{nm}(b)}{mn\textnormal{det}(M_{nm})}C_n \right]^2\\
& \lesssim |a_1|^2+  \sum_{n=2}^{+ \infty} \frac{(mn)^{2(k-1)}}{\varepsilon^{2(nm+k-1)}}\frac{\left(1+\textnormal{log}(nm)\right)^2}{\textnormal{det}(M_{nm})^2} \left[S_{nm}^2 A_n^2+\Lambda_{nm} (b)^2 C_n^2 \right] \\
& \lesssim |a_1|^2+ \sum_{n=2}^{+ \infty}  \frac{(mn)^{2(k-1)}}{\varepsilon^{2(nm+k-1)}} \left(A_n^2+C_n^2 \right)  \\
&\lesssim   \Vert g_1 \Vert_{Y^{k-1}_m} + \Vert g_2 \Vert_{Y^{k-1}_m}.
\end{align*}
We have used the asymptotics $S_{nm}\thicksim \textnormal{log}(nm)$ and $|\textnormal{det}(M_{nm})| \thicksim S_{nm}^2$.\\
$(4)$
We have
\[\partial_{\Omega}DG(\Omega_m^{\pm},0,0)v_{0,m}=\frac{im}{2}\left(\begin{array}{c}
\Omega+\frac{S_m}{b}-\Lambda_1(b) \\ 
-b \Lambda_m(b)
\end{array}  \right) (\omega^m-\overline{\omega}^m).\]
 We resort to reductio  ad absurdum and we suppose that
\[\partial_{\Omega}DG(\Omega_m^{\pm},0,0)v_{0,m}\in \textnormal{Im}(DG(\Omega_m^{\pm},0,0)).\]
Then there exists $(a_1,c_1)\in \R^2$ such that
\[ \left(\begin{array}{c}
\Omega+\frac{S_m}{b}-\Lambda_1(b) \\ 
-b\Lambda_m(b)
\end{array}  \right) =M_m \left( \begin{array}{c}
a_1 \\ 
c_1
\end{array} \right).\]

As $M_m$  has  a one-dimension kernel, $\left(\begin{array}{c}
\Omega+\frac{S_m}{b}-\Lambda_1(b) \\ 
-\Lambda_m(b)
\end{array}  \right)$ will be a scalar multiple of one column of the matrix $M_m$ which happens if and only if
\begin{equation}\label{trans}
 \big(\Omega+S_m-\Lambda_1(b)\big)^2-b^2\Lambda_m(b)^2=0.
\end{equation}
Combining this equation with $\textnormal{det}(M_m)=0$, we get
\[\big(\Omega-S_m+b^2\Lambda_1(b) \big) \left(\Omega+\frac{S_m}{b}-\Lambda_1(b)\right)+\left(\Omega+\frac{S_m}{b}-\Lambda_1(b)\right)^2=0.\]
This yields
\[\left(\Omega+\frac{S_m}{b}-\Lambda_1(b)\right) \left(2\Omega+(b^2-1)\Lambda_1(b)+\left(-1+\frac{1}{b}\right)S_m\right)=0\]
which is equivalent to
\[\Omega+\frac{S_m}{b}-\Lambda_1(b)=0 \text{ ou } \Omega=\frac{1}{2}\left((1-b^2)\Lambda_1(b)+(1-\frac{1}{b})S_m\right).\]
This first possibility is excluded by $(\ref{trans})$ because  $\Lambda_m(b)\neq 0$ and the second one is also impossible because it corresponds to a double  eigenvalue which is not also the case here. We obtain an absurdity and this concludes the proof of Theorem $\ref{appli}$.
\end{proof}

\end{document}